\documentclass[a4paper,reqno]{amsart}

\usepackage{amsmath, amssymb, amsthm,mathtools,stmaryrd}
\usepackage{thmtools}
\usepackage{cancel}
\usepackage[normalem]{ulem}
\usepackage{hyperref}
\usepackage{cleveref}
\usepackage[english]{babel}
\usepackage{enumitem}

\usepackage[all]{xy}
\usepackage{graphicx}
\usepackage{mathtools}
\usepackage{mathrsfs}

\usepackage{verbatim}

\usepackage{color}
\usepackage{tikz-cd}
\tikzcdset{arrow style=tikz, diagrams={>=stealth}}

\usepackage[most]{tcolorbox}
  {
\begin{tcolorbox}[breakable,
 enhanced jigsaw,
 opacityback=0,
 sharp corners,
 parbox=false,
 boxrule=0mm,
 top=0mm,bottom=0pt,left=0pt,right=0pt,
 boxsep=0pt,
 frame hidden,
 parbox=false,
 before upper=\indent\strut,
 before=\par,after=\par,
 finish={\draw[thick,red] ([xshift=-0.4\textwidth]frame.north)--([xshift=0.4\textwidth]frame.north)--([xshift=-0.4\textwidth]frame.south)--([xshift=0.4\textwidth]frame.south);}
]}
{\end{tcolorbox}}

\allowdisplaybreaks

\DeclareMathOperator{\Hom}{Hom}
\DeclareMathOperator{\End}{End}

\newcommand*{\id}{\textup{id}}

\numberwithin{equation}{section}

\theoremstyle{plain}

\newtheorem{thm}{Theorem}[section]
\newtheorem{lem}[thm]{Lemma}
\newtheorem{prop}[thm]{Proposition}
 \newtheorem{cor}[thm]{Corollary}
\newtheorem{defi}[thm]{Definition}
\newtheorem{LD}[thm]{Lemma and Definition}

\theoremstyle{remark}

\newtheorem{rem}[thm]{Remark}

\numberwithin{equation}{section}
\newcommand\inv{^{-1}}

\newcommand{\ot}{\otimes}

\newcommand\ol{\overline}

\newcommand{\beq}{\begin{equation}}
\newcommand{\eeq}{\end{equation}}

\newcommand{\Bi}{\textup{BiGal}}

\newcommand{\cL}{\mathcal{L}}
\newcommand{\cR}{\mathcal{R}}

\newcommand{\BB}{\overline{B}}
\newcommand{\M}{\mathcal{M}}

\newcommand{\ltau}{\overset\leftharpoonup\tau}
\newcommand{\rtau}{\overset\rightharpoonup\tau}
\newcommand{\lcan}{\overset\leftharpoonup\can}
\newcommand{\rcan}{\overset\rightharpoonup\can}

\newcommand{\one}[1]{{#1}{}_{\scriptscriptstyle{(1)}}}
\newcommand{\two}[1]{{#1}{}_{\scriptscriptstyle{(2)}}}
\newcommand{\three}[1]{{#1}{}_{\scriptscriptstyle{(3)}}}

\newcommand{\tuno}[1]{{#1}{}{}^{\scriptscriptstyle{<1>}}}
\newcommand{\tdue}[1]{{#1}{}{}^{\scriptscriptstyle{<2>}}}

\newcommand{\yi}[1]{{#1}{}{}^{\scriptscriptstyle{[1]}}}
\newcommand{\er}[1]{{#1}{}{}^{\scriptscriptstyle{[2]}}}

\newcommand{\teins}[1]{{#1}{}{}^{\scriptscriptstyle{(1)}}}
\newcommand{\tzwei}[1]{{#1}{}{}^{\scriptscriptstyle{(2)}}}

\newcommand{\ti}{{\tilde{i}}}
\newcommand{\trho}{{\tilde{\rho}}}

\newcommand{\lbiprod}{{>\!\!\!\triangleleft\kern-.33em\cdot}}
\newcommand{\rbiprod}{{\cdot\kern-.33em\triangleright\!\!\!<}}

\newcommand{\z}{{}_{\scriptscriptstyle{(0)}}}
\newcommand{\rz}{{}_{\scriptscriptstyle{[0]}}}
\renewcommand{\o}{{}_{\scriptscriptstyle{(1)}}}
\newcommand{\ro}{{}_{\scriptscriptstyle{[1]}}}
\newcommand{\mo}{{}_{\scriptscriptstyle{(-1)}}}
\newcommand{\rmo}{{}_{\scriptscriptstyle{[-1]}}}
\renewcommand{\t}{{}_{\scriptscriptstyle{(2)}}}

\newcommand{\mt}{{}_{\scriptscriptstyle{(-2)}}}
\newcommand{\rmt}{{}_{\scriptscriptstyle{[-2]}}}
\renewcommand{\th}{{}_{\scriptscriptstyle{(3)}}}

\newcommand{\rmth}{{}_{\scriptscriptstyle{[-3]}}}
\newcommand{\fo}{{}_{\scriptscriptstyle{(4)}}}

\newcommand{\di}{{\diamond_{B}}}
\newcommand\ou[1]{\otimes_{#1}}
\newcommand{\la}{{\triangleright}}

\DeclareMathOperator{\tens}{\otimes}
\newcommand{\CR}{\mathcal{R}}
\newcommand{\CH}{\mathcal{H}}
\newcommand{\CM}{\mathcal{M}}
\newcommand{\CL}{\mathcal{L}}

\newcommand{\can}{{\rm can}}

\newcommand{\op}{{\operatorname{op}}}

\newcommand\co{{\operatorname{co}}}
\newcommand\LComod[1]{{^{#1}}\mathcal M}
\newcommand\RMod[1]{\mathcal M_{#1}}
\newcommand\BiMod[1]{{}_{#1}\mathcal M_{#1}}
\newcommand\cop{{\operatorname{cop}}}

\newcommand\skpr{\varoslash}
\newcommand\askpr{\boxslash}
\newcommand\lcp{\#}
\begin{document}

\author{Xiao Han}
\address[]{\textit{Xiao Han},
Queen Mary University of London.
}
\email{x.h.han@qmul.ac.uk}
\author{Peter Schauenburg}
\address{\textit{Peter Schauenburg}, Université Bourgogne Europe, CNRS, IMB UMR 5584, 21000 Dijon, France}
\email{peter.schauenburg@ube.fr}


\title{Cleft Extensions for Hopf Algebroids without Antipodes}

\begin{abstract}
We introduce cleft extensions for Hopf algebroids in the sense of \cite{schau1} and \cite{schau3}. We prove the equivalence between cleft extensions, $\sigma$-twisted crossed products, and Hopf--Galois extensions with the normal basis property, thereby generalizing the theory of cleft extensions for Hopf algebroids developed in \cite{BB} and fitting in with the general theory of Galois and biGalois extensions over Hopf algebroids developed in \cite{HS25}. 

We investigate the Ehresmann Hopf algebroid associated with a cleft extension and show that it is isomorphic to a generalized version of the Connes--Moscovici Hopf algebroid. A special case of the Connes-Moscovici Hopf algebroid, namely the case where the coinvariants of the cleft extension coincide with the base of the Hopf algebroid, is a Drinfeld twist of a Hopf algebroid by a two-cocycle, generalizing \cite{Boehm,HM22,HM23}.
\end{abstract}

\maketitle

\section{Introduction}

In analogy with the relationship between Hopf algebras and groups, Hopf algebroids may be regarded as algebraic or quantum counterparts of groupoids. There are several definitions of Hopf algebroids in the literature. In this paper, we work with Hopf algebroids in the sense of \cite{schau1} and \cite{schau3}. In particular, our Hopf algebroids are not assumed to admit antipodes; rather they fulfill two conditions that together are equivalent to having a bijective antipode in the case of ordinary bialgebras. 

Hopf-Galois extensions are understood as noncommutative-geometric versions of the notion of a principal fiber bundle \cite{schneider}. Namely if a Hopf algebra $H$ is thought of as the function algebra on a quantum group, then an $H$-Galois extension  $N\subset P$ should be thought of as the function algebras $P$ on the total space and $N$ on the base space of a quantum principal fiber bundle with that structure quantum group.

Hopf-Galois extensions over an ordinary Hopf algebra already give rise to Hopf algebroids, namely the Ehresmann Hopf algebroid $L(P,H)$ constructed in \cite{schau4}. In the sense alluded to above it can be interpreted as the function algebra on the gauge groupoid of the extension \cite{HL20}. It was already observed in \cite{schau4} that $P$ has some properties making it similar to an $L(P,H)$-$H$-biGalois extension, thus generalizing \cite{schau5}. However, lacking a general Galois theory for Hopf algebroids this was only an analogy. A full theory of bi-Galois extensions over Hopf algebroids, including an Ehresmann Hopf algebroid construction, was given in \cite{HS25}.

For an ordinary Hopf algebra, cleft extensions are a well-studied particular class of Hopf-Galois extensions \cite{BM,DT, mont, schau7}. Purely algebraically, they are extensions whose properties are defined by maps satisfying suitable equations. Going up, a cleft extension $P$ can be constructed as a crossed product $P=N\#_\sigma H$ by giving a twisted multiplication on $N\ot H$ in terms of a ``weak action" of $H$ on $N$ and a ``cocycle" on $H$ with values in $N$. The weak action and cocycle conditions are algebraic equations on multilinear maps. Going down, a cleft comodule algebra $H$ is defined by a "cleaving" $j\colon H\to P$ subject to suitable equations as well (it is colinear and admits a convolution inverse). Such a map allows to cleave $P$ into a tensor product $N\ot H$ whose structure is that of a crossed product, and which is a Hopf-Galois extension. In terms of noncommutative algebraic geometry, cleft extensions can be interpreted as trivial principal bundles \cite{AFL,BHMS,HRZ,HKMZ,Pflaum}.

The present paper seeks to give a suitable theory of cleft extensions over a Hopf algebroid $\CH$ instead of an ordinary Hopf algebra. We can already remark at this point that such a theory was developed by Böhm and Brzezinski in \cite{BB}, but only for a more restricted class of Hopf algebroids that admit antipodes. We give a detailed comparison in an appendix and use some results from \cite{BB}.

In the general case, the principal problem with the generalization of cleft extensions already arises with the very definition. In the classical case the definition requires a convolution invertible colinear map $\CH\to P$, but for general bialgebroids there is no convolution product that would be applicable to write this definition. This already concerns the definition of Hopf algebroids themselves; in the classical case a bialgebra is a Hopf algebra if the identity is convolution invertible, for a bialgebroid the underlying convolution does not exist. Thus already the most trivial cleft extension over $\CH$, namely $\CH$ itself, does not generalize naively to the Hopf algebroid case. 

The two conditions replacing the existence of a bijective antipode in the definition of a Hopf algebroid are bijectivity of the maps (conventions and notations on the diverse $B$-module structures to be recalled in \Cref{sec2})
\begin{align*}
    \lambda\colon \CH\ou\BB\CH&\to \CH\di\CH&\mu\colon\CH\ou B\CH&\to \CH\di\CH\\
    X\ot Y&\mapsto X\o\ot X\t Y&X\ot Y&\mapsto X\o Y\ot X\t.
    \intertext{In places where Hopf algebra theory uses the antipode, the partial inverses}
    X_+\ot X_-&\mapsfrom X\ot 1& X_{[+]}\ot X_{[-]}&\mapsfrom 1 \ot X
\end{align*}
then make frequent appearances.

The replacement we now propose for the classical notion of cleftness for a right $\CH$-comodule algebra $P$ is to require the existence of colinear maps 
\[
\rho,\trho:\cL\to P
\]
such that $\rho$ is right $B$-linear, $\trho$ is left $B$-linear, and
\[
\trho(X_{+})\rho(X_{-})=\varepsilon(X)1_{P},\qquad
\rho(X_{[+]})\trho(X_{[-]})=\overline{\varepsilon(X)}1_{P}.
\]
Thus we require a sort of convolution invertibility that uses the restricted inverses of the $\mu,\lambda$ maps instead of comultiplication. We will discuss properties of this twisted convolution at length, and in particular prove that it generalizes ordinary convolution invertibility in the classical situation (with $\trho$ replacing the classical cleaving map). For the case $P=\CH$ it should be noted that the existence of such a cleaving is tautological: Both $\rho$ and $\trho$ can be chosen to be the identity. The new  condition thus does not introduce a cleaving independent of the trivial cleaving of $\CH$ itself, but rather postulates a specific way of extending that cleaving to the comodule algebra.

As stated, the condition to be cleft 
is formulated (as in the classical case) as a set of equations (colinearity, mutual "inverse") on linear maps. We will also require the existence of a reverse comodule structure on $P$, which means another map (said comodule structure) and equations.

In case $\CH$ is suitably faithfully flat over its base, Chemla \cite{C20} has shown that all comodule structures are reversible. Under the condition that $P$ is left faithfully flat over the base $\ol B$ of $\CH$, we will see that invertibility of $\rho$ in our twisted sense is equivalent to bijectivity of a certain canonical map $\CH\ou B P\to P\di\CH.$ We will prefer, however, to develop as much as possible of the theory without imposing flatness conditions, to underline the ``explicit" nature of the theory of cleft extensions, where everything is defined by explicit equations, whereas in the general theory of (bi)Galois extensions notably the invariant subalgebra is only given as an equalizer (for a cleft extension it turns out to be also the image of a module projection) and flatness conditions over various subrings are frequently needed to deduce the existence or bijectivity of crucial maps. In our brief summary of biGalois theory above we have summarily swept all such technical problems
under the rug.

With our definition in place, we can prove a complete analog of the classical characterization of cleft extensions known in the Hopf algebra case. Namely, an $\CH$-anti-cleft comodule algebra $P$ is the same thing as an anti-Galois extension with a normal basis $P\cong N\ou B\CH$ and also the same thing as a crossed product comodule algebra $P\cong N\#_\sigma\CH$ with invertible cocycle. For the equivalence between crossed products and comodule algebras with a normal basis we can tap into \cite{BB} except for the crucial link between the Galois condition and invertibility of the cocycle.

Having thus established a theory of cleft extensions, we turn to their biGalois theory in the sense developed in \cite{HS25}. That is to say, we consider the Ehresmann Hopf algebroid $L(N\#_\sigma\CH,\CH)$ of an anti-cleft extension and the inverse of the extension in the groupoid of bi-Galois extensions. All these notions are defined and studied in \cite{HS25} under diverse faithful flatness conditions both over the base of the Hopf algebroid and the coinvariant subalgebra (which is the base of the Ehresmann Hopf algebroid). In our present situation, we can compute them explicitly in terms of the data determining the cleft extension. Notably, the inverse in the groupoid of bi-Galois extensions of a crossed product right comodule algebra with invertible cocycle turns out to be the crossed product left comodule algebra using the inverse cocycle. The Ehresmann algebroid is a triple tensor product $N\#_\sigma\CH\#_\sigma N$ of $\CH$ with two copies of the coinvariant subalgebra of $P$, with a twice twisted multiplication using the action (twice) and the cocycle and its inverse; this is a Hopf algebroid generalization of the so called Connes-Moscovici Hopf algebroid constructed from a crossed product over a Hopf algebra (see \cite{connes-moscovici}, the algebraic construction appears without a cocycle in \cite{Takeuchi77} and with a cocycle in \cite{schau4}).

Given the more explicit nature of cleft extensions as compared with general Galois extensions, we can also establish some of these constructions \emph{without} imposing the flatness conditions needed in the general theory in \cite{HS25}. Up to some subtleties (mostly coming from the fact that subspaces in tensor products defined by equalizers may not coincide with the "obvious" candidate in the absence of flatness conditions) we can establish the same results, notably the construction of the generalized Ehresmann or Connes-Moscovici Hopf algebroid.

A special case arises when the coinvariant subring of the extension coincides with the base of the Hopf algebroid. In the case of ordinary Hopf algebras \cite{schau5} such extensions are sometimes called Hopf Galois \emph{objects}. They can be identified with $\CH$ itself with multiplication twisted by a cocycle, and their Ehresmann Hopf algebroid is $\CH$ with multiplication twisted on both sides with a cocycle and its inverse. We describe this as a special case of the Ehresmann construction and find a cocycle twist as described in \cite{Boehm,HM22,HM23}, except that we can simply do without one condition imposed on cocycles there. We study examples of such cocycle twists arising from considering two crossed products with the same Hopf algebra but different cocycles and actions; in particular this gives examples of twists that do not fulfill the conditions in \cite{Boehm,HM22,HM23}.

\subsection*{Acknowledgements} Xiao Han was supported by the European Union's Horizon 2020 research and innovation program under the Marie Sklodowska-Curie grant agreement No 101027463 at the early stage of the project. Xiao Han was supported by Leverhulme Trust project grant RPG-2024-177 at the later stage of the project. Xiao Han was supported by COST Action CA21109 for Short-Term Scientific Mission. Xiao Han is grateful to the Institut de Mathematiques de Bourgogne
(Dijon) for hospitality.

\section{Preliminaries} \label{sec2}

In this section, we will recall some  definitions and notation concerning Hopf algebroids and the general theory of (bi)Galois extensions. We do this to make the paper more self-contained; the results either come from \cite{HS25} or from the literature cited there, with the exception of \Cref{skewedbijections}.

Let $B$ be an unital algebra over a field $k$. We denote the opposite algebra by $\BB$ and let $B\to \BB$, $b\mapsto\Bar{b}$ for any $b\in B$ be the obvious $k$-algebra antiisomorphism. Define $B^{e}:=B\ot \BB$, so $B$ and $\BB$ are obvious subalgebras of $B^{e}$. Let $M, N$ be $B^{e}$-bimodules. We define
\begin{align*}
    M\di N:=\int_{b} {}_{\Bar{b}}M\ot {}_{b}N:=&M\ot N/\langle \Bar{b}m\ot n-m\ot bn|b\in B, m\in M, n\in N\rangle\\
    M\blacklozenge_{B} N:=\int_{b} M{}_{b}\ot N{}_{\Bar{b}}:=&M\ot N/\langle m\ot n\Bar{b}-mb\ot n|b\in B, m\in M, n\in N\rangle\\
    M\ot_{B} N:=\int_{b} M_{b}\ot {}_{b}N:=&M\ot N/\langle mb\ot n-m\ot bn|b\in B, m\in M, n\in N\rangle\\
    M\ot_{\BB} N:=\int_{b} M_{\Bar{b}}\ot {}_{\Bar{b}}N:=&M\ot N/\langle m\Bar{b}\ot n-m\ot \Bar{b}n|b\in B, m\in M, n\in N\rangle\\
\end{align*}
For convenience, we also define $N\ot^{B}M=\int_{b} {}_{b}N\ot M_{b}$ and  $N\ot^{\BB}M=\int_{b} {}_{\Bar{b}}N\ot M_{\Bar{b}}$. Moreover, we define
\begin{align*}
    \int^{b}M_{\Bar{b}}\ot N_{b}:=\{\sum_{i}m_{i}\ot n_{i}\in M\ot N\quad |\quad m_{i}\Bar{b}\ot n_{i}=m_{i}\ot n_{i}b, \forall b\in B\}.
\end{align*}

The symbol $\int^{b}$ and $\int^{c}$ commute, also $\int_{b}$ and $\int_{c}$ commute. However, in general, the symbol $\int^{b}$ and $\int_{c}$ doesn't commute. For any $\BB$-bimodule $M$ and any $B$-bimodule $N$, we also define
\begin{align*}
    M\times_{B}N:=\int^{a}\int_{b} {_{\Bar{b}}M_{\Bar{a}}}\ot {_bN_a}.
\end{align*}
$M\times_{B}N$ is called the Takeuchi product of $M$ and $N$. If $P$ is a $B^{e}$-bimodule, then $P\times_{B}N$ is a $B$-bimodule with $B$ acting on $P$. Similarly, $M\times_{B}P$ is a $\BB$-bimodule with $\BB$ acting on $P$. If both $M$ and $N$ are $B^{e}$-bimodule, then $M\times N$ is also a $B^{e}$-bimodule. However, the product $\times_{B}$ is neither associative and unital on the category of $B^{e}$-bimodules. For any $M,N, P\in {}_{B^{e}}\M_{B^{e}}$, we can define
\begin{align*}
    M\times_{B}P\times_{B}N:=\int^{a,b}\int_{c,d} {}_{\Bar{c}}M_{\Bar{a}}\ot {}_{c,\Bar{d}}P_{a, \Bar{b}}\ot {}_{d}N_{b},
\end{align*}
where $\int^{a,b}:=\int^{a}\int^{b}$ and $\int_{c,d}:=\int_{c}\int_{d}$. There are maps
\begin{align*}
    &\alpha:(M\times_{B}P)\times_{B} N\to M\times_{B}P\times_{B}N,\quad m\ot p\ot n\mapsto m\ot p\ot n,\\
    &\alpha':M\times_{B}(P\times N)\to M\times_{B}P\times_{B}N,\quad m\ot p\ot n\mapsto m\ot p\ot n.\\
\end{align*}
Notice that neither $\alpha$ nor $\alpha'$ are isomorphisms  in general. In particular, if all $B$-module and $\BB$-module structures are faithfully flat, then $\alpha$ and $\alpha'$ are isomorphisms.

\subsection{Hopf algebroids}

Here we recall the basic definitions (cf. \cite{BW}, \cite{Boehm}). Let $B$ be a unital algebra over a field $k$.
A {\em $B$-ring}   means a unital algebra in the monoidal category ${}_B\CM_B$ of $B$-bimodules. Likewise,  a  {\em $B$-coring} is a coalgebra in ${}_B\CM_B$. Morphisms of $B$-(co)rings are defined to be morphisms of (co)algebras, but in the category ${}_B\CM_B$.

Specifying a unital $B$-ring $\cL$ is equivalent to specifying a unital algebra $\cL$ (over $k$) and an algebra map $\eta:B\to \cL$. Left and right multiplication in $\cL$ pull back to left and right $B$-actions as a bimodule (so $bXc=\eta(b)X\eta(c)$ for all $b,c\in B$ and $X\in \cL$) and the product descends to the product $\mu_B:\cL\tens_B\cL\to \cL$ with $\eta$ the unit map. Conversely, given
$\mu_B$ we can pull back to an associative product on $\cL$ with unit $\eta(1)$.

Now suppose that $s:B\to \cL$ and $t:\BB\to \cL$ are algebra maps with images that commute. Then $\eta(b\tens c)=s(b)t(c)$ is an algebra map $\eta: B^e\to \cL$, where $B^e=B\tens \BB$, and is equivalent to making $\cL$ a $B^e$-ring. The left $B^e$-action part of this is equivalent to a $B$-bimodule structure
\begin{equation}\label{eq:rbgd.bimod}
b.X.c=b\Bar{c}X:= s(b) t(c)X
\end{equation}
for all $b,c\in B$ and $X\in \cL$.

If $\cL$ and $\mathcal R$ are two $B^e$-rings, then the Takeuchi product $\cL\times_B\mathcal R$ is an algebra with multiplication given by $(\ell\otimes r)(\ell'\otimes r')=\ell\ell' \otimes rr'$ for $\ell\ot r,\ell'\ot r'\in \cL\times_B\mathcal R$.

\begin{defi}\label{def:left.bgd} Let $B$ be a unital algebra. A left $B$-bialgebroid (or left bialgebroid over $B$) is an algebra $\cL$ and (`source' and `target') commuting algebra maps $s:B\to \cL$ and $t:\BB\to \cL$ (thus making $\cL$ a $B^e$-ring), such that
\begin{itemize}
\item[(i)]  There are two left $B^{e}$-module maps, the coproduct $\Delta:\cL\to \cL\times_{B} \cL$ and counit $\varepsilon:\cL\to B$ satisfying
\begin{align*}
&\alpha\circ(\Delta\times_{B}\id)\circ\Delta=\alpha'\circ(\id\times_{B}\Delta)\circ\Delta:\cL\to \cL\times_{B}\cL\times_{B}\cL\\    &\varepsilon(X\o)X\t=\overline{\varepsilon(X\t)}X\o=X,
\end{align*}
for any $X\in\cL$, where we use the sumless Sweedler index to denote the coproduct, namely, $\Delta(X)=X\o\ot X\t$.
\item[(ii)] The coproduct is an algebra map. The counit $\varepsilon$ is a left character in the following sense:
\begin{equation*}\varepsilon(1_{\cL})=1_{B}, \quad \varepsilon(X\varepsilon(Y))=\varepsilon(XY)=\varepsilon(X\overline{\varepsilon(Y)})\end{equation*}
for all $X,Y\in \cL$ and $a\in B$.
\end{itemize}
\end{defi}

\begin{rem}
   Condition (i) imply $\cL$ is a $B$-coring with the $B$-bimodule structure given by (\ref{eq:rbgd.bimod}). By condition (ii), we can define an algebra map $\hat{\varepsilon}: \cL\to \operatorname{End}(B)$ by $\hat{\varepsilon}(X)(b):=\varepsilon(X b)$. Moreover, the image of $\alpha\circ(\Delta\times_{B}\id)\circ\Delta=\alpha'\circ(\id\times_{B}\Delta)\circ\Delta$ could be also denoted by the Sweedler index, namely, $\alpha\circ(\Delta\times_{B}\id)\circ\Delta(X)=\alpha'\circ(\id\times_{B}\Delta)\circ\Delta(X)=X\o\ot X\t\ot X\th\in \cL\times_{B}\cL\times_{B}\cL$. Unlike the bialgebra case, the elements $X\o\o\ot X\o\t\ot X\t$, $X\o\ot X\t\o\ot X\t\t$ and $X\o\ot X\t\ot X\th$ are not the same, since they belong to three different vector spaces, even if in many interesting cases both $\alpha$ and $\alpha'$ are isomorphisms.
   Morphisms between left $B$-bialgebroids are $B$-coring maps which are also $B^e$-ring map maps.
\end{rem}


\begin{defi}\label{defHopf}
A left bialgebroid $\cL$ over $B$ is a left Hopf algebroid (\cite{schau1}, Thm and Def 3.5.) if
\[\lambda: \cL\ot_{\BB}\cL\to \cL\di\cL,\quad
    \lambda(X\ot Y)=\one{X}\ot \two{X}Y\]
is invertible.\\
Similarly, A left bialgebroid $\cL$ over $B$ is a anti-left Hopf algebroid if
\[\mu: \cL\ot_{B}\cL\to \cL\di \cL,\quad
    \mu(X\ot Y)=\one{X}Y\ot \two{X}\]
is invertible. A left $B$-bialgebroid is a $B$-Hopf algebroid if it is a left Hopf algebroid and anti-left Hopf algebroid.
\end{defi}
In the following we will always use the balanced tensor product explained as above. If $B=k$ then this reduces to the map $\cL\tens\cL\to \cL\tens\cL$ given by $h\tens g\mapsto h\o\tens h\t g$ which for a usual bialgebra has an inverse, namely $h\tens g\mapsto h\o\tens (Sh\t)g$ if and only if there is an antipode. We adopt the shorthand
\begin{equation}\label{X+-} X_{+}\ot_{\BB}X_{-}:=\lambda^{-1}(X\di 1)\end{equation}
\begin{equation}\label{X[+][-]} X_{[+]}\ot_{B}X_{[-]}:=\mu^{-1}(1\di X).\end{equation}
So for a Hopf algebra $H$, $h_{+}\ot h_{-}=h\o\ot S(h\t)$ and $h_{[-]}\ot h_{[+]}= S^{-1}(h\o)\ot h\t$ for any $h\in H$.
We recall from \cite[Prop.~3.7]{schau1} that for a left Hopf algebroid, and any $X, Y\in \cL$ and $a,a',b,b'\in B$,
\begin{align}
    \one{X_{+}}\di{}\two{X_{+}}X_{-}&=X\di{}1\label{equ. inverse lamda 1};\\
    \one{X}{}_{+}\ot_{\BB}\one{X}{}_{-}\two{X}&=X\ot_{\BB}1\label{equ. inverse lamda 2};\\
    (XY)_{+}\ot_{\BB}(XY)_{-}&=X_{+}Y_{+}\ot_{\BB}Y_{-}X_{-}\label{equ. inverse lamda 3};\\
    1_{+}\ot_{\BB}1_{-}&=1\ot_{\BB}1\label{equ. inverse lamda 4};\\
    \one{X_{+}}\di{}\two{X_{+}}\ot_{\BB}X_{-}&=\one{X}\di{}\two{X}{}_{+}\ot_{\BB}\two{X}{}_{-}\label{equ. inverse lamda 5};\\
    X_{+}\ot\one{X_{-}}\ot{}\two{X_{-}}&=X_{++}\ot X_{-}\ot{}X_{+-}\in\int^{a,b}\int_{c,d}{}_{\Bar{a}}\cL_{\Bar{c}}\ot {}_{\Bar{d}}\cL_{\Bar{b}}\ot {}_{\Bar{c},d}\cL_{b,\Bar{a}}
    \label{equ. inverse lamda 6};\\
    X&=X_{+}\overline{\varepsilon(X_{-})}\label{equ. inverse lamda 7};\\
    X_{+}X_{-}&=\varepsilon(X)\label{equ. inverse lamda 8};\\
    (a\Bar{a'}Xb\Bar{b'})_{+}\ot_{\BB}(a\Bar{a'}Xb\Bar{b'})_{-}&=aX_{+}b\ot_{\BB}b'X_{-}a'\label{equ. inverse lamda 9};\\
    \Bar{b}X_{+}\ot_{\BB}X_{-}&=X_{+}\ot_{\BB}X_{-}\Bar{b}\label{equ. inverse lamda 10}.
\end{align}
Similarly, for an anti-left Hopf algebroid, from \cite{BS} we have
\begin{align}
    \one{X_{[+]}}X_{[-]}\di{}\two{X_{[+]}}&=1\di{}X\label{equ. inverse mu 1};\\
    \two{X}{}_{[+]}\ot_{B}\two{X}{}_{[-]}\one{X}&=X\ot_{B}1\label{equ. inverse mu 2};\\
    (XY)_{[+]}\ot_{B}(XY)_{[-]}&=X_{[+]}Y_{[+]}\ot_{B}Y_{[-]}X_{[-]}\label{equ. inverse mu 3};\\
    1_{[+]}\ot_{B}1_{[-]}&=1\ot_{B}1\label{equ. inverse mu 4};\\
    (\one{X_{[+]}}\di{}\two{X_{[+]}})\ot_{B}X_{[-]}&=(\one{X}{}_{[+]}\di{}\two{X})\ot_{B}\one{X}{}_{[-]}\label{equ. inverse mu 5};\\
    \one{X_{[-]}}\ot{}\two{X_{[-]}}\ot X_{[+]}&=X_{[+][-]}\ot{}X_{[-]}\ot X_{[+][+]}\in \int^{a,b}\int_{c,d}{}_{c,\Bar{d}}\cL_{a,\Bar{b}}\ot {}_{d}\cL_{b}\ot {}_{a}\cL_{c}\label{equ. inverse mu 6};\\
    X&=X_{[+]}\varepsilon(X_{[-]})\label{equ. inverse mu 7};\\
    X_{[+]}X_{[-]}&=\overline{\varepsilon(X)}\label{equ. inverse mu 8};\\
    (a\Bar{a'}Xb\Bar{b'})_{[+]}\ot_{B}(a\Bar{a'}Xb\Bar{b'})_{[-]}&=\Bar{a'}X_{[+]}\Bar{b'}\ot_{B}\Bar{b}X_{[-]}\Bar{a}\label{equ. inverse mu 9};\\
    X_{[+]}\ot_{B}X_{[-]}b&=bX_{[+]}\ot_{B}X_{[-]}\label{equ. inverse mu 10}.
\end{align}

\begin{prop}\label{prop. anti left and left Galois maps}
    If $\cL$ is a left and anti-left Hopf algebroid, then we have
    \begin{align}\label{mixed coass 1}
        X_{[+]}\ot X_{[-]+}\ot X_{[-]-}&=X\t{}_{[+]}\ot X\t{}_{[-]}\ot X\o\in \cL\ot_{B}\cL\ot_{\BB}\cL;\\\label{mixed coass 2}
        X_{+}\ot X_{-[+]}\ot X_{-[-]}&=X\o{}_{+}\ot X\o{}_{-}\ot X\t\in \cL\ot_{\BB}\cL\ot_{B}\cL;\\\label{mixed coass 3}
        X_{[+]+}\ot X_{[-]}\ot X_{[+]-}&=X_{+[+]}\ot X_{+[-]}\ot X_{-}\in \int^{c,d}\int_{a,b} {}_{c, \Bar{d}}\cL_{\Bar{a}, b}\ot {}_{b}\cL_{c}\ot {}_{\Bar{a}}\cL_{\Bar{d}},
    \end{align}
   for any $X\in\cL$.
\end{prop}
\begin{proof}
    The first equality can be proved by comparing the result of $\id_{\cL}\ot \lambda $ on both sides. The second
 equality can be proved by comparing the result of $\id_{\cL}\ot \mu$ on both sides. The third
 equality can be proved by comparing the result of $\mu\ot \id_{\cL}$ on both sides.
\end{proof}

\subsection{Comodules of left bialgebroids}

Let $\CL$ be a left bialgebroid over $B$. In particular $\CL$ is a $B$-coring with respect to the left $B$-module and left $\ol B$-module structure. Thus we can define a comodule of $\CL$ to be the same as a comodule over that coring, that is $P\in {}_B\mathcal M$ with a coassociative coaction $\delta_0\colon P\to\CL\di P$. We will use the notation $\delta_0(p)=p\o\ot p\z$.

If $P$ is a left $\CL$-comodule, then by \cite{HHP}, there is a unique right $B$-module structure on $P$, namely $pb=\varepsilon(p\mo b)p\z$, such that $\operatorname{Im}(\delta_0)\subset \CL\times_BP$. Thus we can define $\delta\colon P\to\CL\times_BP$ and $P$ is a comodule according to the definition in \cite{schau1}. We will not distinguish between $\delta$ and $\delta_0$ in the sequel.

Similarly a right comodule is defined by a coaction $\delta\colon P\to P\di\CL$. It is a $\BB$-bimodule, and we will write $\delta(p)=p\z\ot p\o$.

\begin{defi}\label{def. regular of comodule}
A left $\cL$-comodule $P$ is called  reversible (\cite{C20} with different terminology), if
\[\phi: P\ot_{B} \cL\to \cL\di P, \quad \phi(X\ot p)=p\mo X\ot p\z\]
is invertible. We adopt the shorthand
\begin{align}
    p\rz\ot_{B} p\ro=\phi^{-1}(1\di p).
\end{align}
A right $\CH$-comodule $P$ is called  reversible right $\CH$-comodule \cite{C20}, if
\[\psi: P\ot_{\BB} \CH\to P\di \CH, \quad \psi(p\ot_{\BB} X)=p\z \di p\o X\]
is invertible. We adopt the shorthand
\begin{align}
    p\rz\ot_{\BB} p\rmo=\psi^{-1}(p \di 1).
\end{align}
\end{defi}

If $P$ is a reversible left $\CL$-comodule, then $\ol P$ is a right $\CL$-comodule, where $\ol P$ is $P$ with the $\BB$-bimodule structure defined by the $B$-bimodule structure of $P$, and the coaction
$p\mapsto p\rz\ot p\ro$. Similarly, a reversible right comodule defines a left comodule structure on the same space, with the $B$-bimodule structure defined by the original $\BB$-bimodule structure.
This defines a bijection between reversible left comodule structures and reversible right comodule structures.

If a comodule is reversible, then there are compatibility conditions linking the comodule structures and the (anti)translation maps:
    \begin{align}\label{equ. regular map 7}
        p\rz\z\ot{}p\rz\o\ot p\rmo=&p\z\ot{} p\o{}_{+}\ot p\o{}_{-},\\
    \label{equ. regular map 8}
        p\z\rz\ot p\z\rmo\ot p\o=&p\rz\ot p\rmo{}_{[+]}\ot p\rmo{}_{[-]},
    \end{align}
    for any $p\in P$.

 We recall the definition of cotensor product:
Let $P$ be a right comodule of a left $B$-bialgebroid $\cL$ and $Q$ be a left comodule of $\cL$, then the cotensor product $P\Box^{\cL} Q$ is defined to be the equalizer
\[P\Box^{\cL} Q:=\{p\ot q\in P\di{}Q\quad|\quad p\z\ot p\o\ot q=p\ot q\mo\ot q\z\}\subset P\times_BQ,\]
where the balanced tensor product $P\di{} Q$ above is induced by the $\BB$-bimodule of $P$ and the $B$-bimodule of $Q$, i.e. $\overline{b}p\di{}q=p\di{}bq$.

\subsection{Hopf (bi)Galois extensions}
The theory of bi-Galois extensions was developed in \cite{HS25}. It often has to require \emph{some} conditions of (faithful) flatness of the extensions as well of the Hopf algebroids involved. For this introduction we are "generous" and sometimes impose more restrictions than strictly necessary. One result of the present paper will be to establish some of the results with practically no such restrictions in the cleft case.

\begin{defi}
Given a left bialgebroid over $B$,    a left $\cL$-comodule algebra $P$ is a $B$-ring and a left $\cL$-comodule, such that the coaction is a $B$-ring map. Let $N={}^{co\cL}P$ be the left invariant subalgebra of $P$, $N\subseteq P$ is called a left $\cL$-Galois extension if the left canonical map $\lcan: P\ot_{N} P\to \cL\di{}P$ given by
    \[\lcan(p\ot_{N}q)=p\mo\di{}p\z q,\]
    is bijective. If $N\subseteq P$ is a
left $\cL$-Galois extension, the inverse of $\lcan$ can be determined by the left translation map:
    \begin{align}
        \ltau:=\lcan^{-1}|_{\cL\di{}1}:\cL\to P\ot_{N}P, \qquad X\mapsto \tuno{X}\otimes_{N}\tdue{X}.
    \end{align}
 If $P$ is a faithfully flat  left and right $N$-module, we call $N\subseteq P$ a faithfully flat $\cL$-Galois extension.
\end{defi}

\begin{prop}
    Let $N\subseteq P$ be a faithfully flat left $\cL$-Galois extension, then we have

\begin{align}\label{equ. translation map 1}
  \tuno{X}\mo \di \tuno{X}\z \ot_{N} \tdue{X} &= X\o\di{}\tuno{X\t}\ot_{N}  \tdue{X\t},\\
\label{equ. translation map 2}
~~ \tdue{X}\mo\ot{}\tuno{X}  \ot \tdue{X}\z &= X_{-}\ot{}\tuno{X_{+}}\ot\tdue{X_{+}}\in \int_{b}{}_{\Bar{b}}\cL\ot P\ot_{N} {}_{b}P,\\
\label{equ. translation map 6}
\tuno{(\Bar{a}X  \Bar{b})}\ot_{N}\tdue{(\Bar{a}X  \Bar{b})}
&=\tuno{X}\ot_{N}b\tdue{X}a,
\end{align}
for any $X, Y\in \cL$, $p\in P$, $n\in N$ and $a, b\in B$. 
\end{prop}

\begin{defi}
Given a left $B$-bialgebroid $\CH$,  a right $\CH$-comodule algebra $P$ is a $\BB$-ring and a right $\CH$-comodule, such that the coaction is a $\BB$-ring map. Let $N:=P{}^{co\CH}$ be the right invariant subalgebra of $P$, $N\subseteq P$ is called a anti-right $\CH$-Galois extension if the anti-right canonical map $\rcan: P\ot_{N} P\to P\di{}\CH$ given by
    \[\rcan(p\ot_{N}q)=p\z q\di{}p\o ,\]
    is bijective. If $N\subseteq P$ is a
anti-right $\CH$-Galois extension, the inverse of $\rcan$ can be determined by the anti-right translation map:
    \begin{align}
       \rtau:=(\rcan)^{-1}|_{1\di{}\CH}:\CH\to P\ot_{N}P, \qquad X\mapsto \yi{X}\otimes_{N}\er{X},
    \end{align}
    namely $\rcan\inv(p\ot X)=\yi X\ot \er Xp$. If $P$ is a faithfully flat left and right $N$-module, we call $N\subseteq P$ a faithfully flat anti $\CH$-Galois extension.
\end{defi}

\begin{prop}\label{prop. left Hopf Galois extension}
    Let $N\subseteq P$ be an anti-right $\CH$-Galois extension, then we have

\begin{align}\label{equ. anti translation map 1}
  \yi{X}\z \ot_{N} \er{X} \ot{} \yi{X}\o &= \yi{X\o}\ot_{N}\er{X\o}\ot{}X\t\in\int_{b}{}_{\Bar{b}}P\ot_{N}P\ot {}_{b}\CH,\\
\label{equ. anti translation map 2}
~~ \yi{X}\ot_{N}\er{X}\z\di{}\er{X}\o &= \yi{X_{[+]}}\ot_{N}\er{X_{[+]}}\di{}X_{[-]},\\
\label{equ. anti translation map 3}
\yi{X}\z \er{X}\di{}\yi{X}\o &= 1\di{}X,\\
\label{equ. anti translation map 4}
    \yi{p\o}\ot_{N}\er{p\o}p\z&=p\ot_{N}1,\\
 \label{equ. anti translation map 4.5}
    mX\yi{}\ot_{N}X\er{}=&
    X\yi{}\ot_{N}X\er{}m,\\   
\label{equ. anti translation map 5}
\yi{(aX b )}\ot_{N}\er{(aXb )}
&=\yi{X}\ot_{N}\Bar{b}\er{X}\Bar{a},\\
\label{equ. anti translation map 6}
\yi{(\Bar{a}X  \Bar{b})}\ot_{N}\er{(\Bar{a}X  \Bar{b})}
&=\Bar{a}\yi{X}\Bar{b}\ot_{N}\er{X},\\
\label{equ. anti translation map 6.5}
\yi{X}\er{X}
&=\overline{\varepsilon(X)},\\
\label{equ. anti translation map 7}
\yi{(XY)}\ot_{N}\er{(XY)}&=\yi{X}\yi{Y}\ot_{N}\er{Y}\er{X},\\
\label{equ. anti translation map 7.5}
\yi{X_{[+]}}\ot_{N} \yi{X_{[-]}}\ot_{N} \er{X_{[-]}}\er{X_{[+]}}&=\yi{X}\ot_{N}\er{X}\ot_{N} 1,
\end{align}
for any $X, Y\in \cL$, $p\in P$, $m\in N$ and $a, b\in B$. If $P$ is a  reversible right $\CH$-comodule, then we have
\begin{align}
    \label{equ. regular anti translation map 1}
    p\rz \yi{p\rmo}\ot_{N} \er{p\rmo}=&1\ot_{N}p,\\
    \label{equ. regular map 3}
        (pq)\rz\ot_{\BB}(pq)\rmo=& p\rz q\rz\ot_{\BB} q\rmo p\rmo\\
    \label{equ. regular anti translation map 2}
    \yi{X}\ot_{N}\er{X}\rz\di \er{X}\rmo=&\yi{X\t}\ot_{N}\er{X\t}\di X\o,\\
    \label{equ. regular anti translation map 3}
    \yi{X}\rz\ot_{N}\er{X}\ot{} \yi{X}\rmo=&\yi{X_{+}}\ot_{N}\er{X_{+}}\ot{} X_{-}\in \int_{b}P_{\Bar{b}}\ot_{N}P\ot {}_{\Bar{b}}\cL,\\
    \label{equ. regular anti translation map 4}
    \yi{X_{+}} \yi{X_{-}}\ot \er{X_{-}}\ot \er{X_{+}}=&1\ot\yi{X}\ot \er{X}.
\end{align}
\end{prop}

\begin{rem}\label{hopf module adjunction}
    For any right $\CH$-comodule algebra $P$ and $N\subset P^{\co\CH}$ we have a pair of adjoint functors
    \begin{align*}
        \mathcal M^{\CH}_P&\rightleftarrows \RMod N\\
        \Gamma&\mapsto \Gamma^{\co\CH}\\
        M\ou NP&\mapsfrom M. 
    \end{align*}
    The unit and counit of adjunction are given by
    \begin{align*}
        M\ni m&\mapsto m\ot 1\in(M\ou NP)^{\co\CH}\\
        \Gamma^{\co\CH}\ou NP\ni \gamma\ot p&\mapsto \gamma p\in \Gamma.
    \end{align*}
   
    If $\CH$ is faithfully flat as a left module over both bases, and $P$ is a left faithfully flat $\CH$-anti-Galois extension of $N=P^{\co\CH}$ the adjunction above is an equivalence by \cite[Corollary 3.27]{HS25}.
\end{rem}

\begin{defi}\label{def. Hopf biGalois extensions}
    Let $B,N$ be two algebras. 
    Let $\cL$ be a $B$-Hopf algebroid  and let $\CH$ an $N$-Hopf algebroid. 
    
    A $\cL$-$\CH$-Hopf biGalois extension is an algebra $P$,\ such that $\ol N\subseteq P$ is a 
        left $\cL$-Galois extension and $B\subseteq P$ is an  anti-right $\CH$-Galois extension, in such a way that the two comodule structures make $P$ a $\cL$-$\CH$-bicomodule. 

    We denote by $\Bi(\cL, \CH)$ the collection of $\cL$-$\CH$-Hopf biGalois extensions that are faithfully flat as left and right modules both over $B$ and $\ol N$. 
\end{defi}

\begin{prop}\label{bigalois imply equivalence}
 If $P\in\Bi(\cL,\CH)$, then we have a monoidal category equivalence
    \[\LComod{\CH}\ni V\mapsto P\Box^{\CH}V\in\LComod{\cL}.\]
\end{prop}

\begin{LD}\label{lem. bigalois given by ehresmann groupoid} 
     Let $\CH$ be a  $B$-Hopf algebroid which is faithfully flat as left $B$ and $\ol B$-module. Let $N\subseteq P$ be a faithfully flat anti-right $\CH$-Hopf Galois extension. The  Ehresmann Hopf algebroid $L(P,\CH):=(P\ot_{\BB}P^{op})^{co\CH}$ is a Hopf algebroid over $N$ with the $N^e$-ring structure:
     \[s(n)=n\ot_{\overline{B}}1,\quad t(n)=1\ot_{\overline{B}}n,\quad (p\ot_{\overline{B}}q)(p'\ot_{\overline{B}}q')=pp'\ot_{\overline{B}}q'q,\]
    for any $p\ot_{\BB}q, p'\ot_{\BB}q'\in L(P, \CH)$.
    And the $N$-coring structure:
    \[\Delta(p\ot_{\BB} q)=p\ot_{\BB}\yi{q\o}\ot_{N}\er{q\o}\ot_{\BB}\,q\z,\qquad \varepsilon(p\ot q)=pq.\]
   Moreover, $P\in\Bi(L(P,\CH),\CH)$ with the left coaction $\delta:P\to L(P, \CH)\diamond_{N}P$ given by
    \[\delta(p)=p\rz\ot p\rmo\yi{}\ot p\rmo\er{}.\]  
\end{LD}

\subsection{Characterizations of colinear maps}\label{skewedbijections}
In this subsection we collect for later use some exotic Hopf algebroid-specific variants of the following well-known ubiquitous tool:
\begin{rem}\label{free comodule bijection}
    Let $C$ be a $B$-coring, $N$ a right $B$-module, and $M$ a right $C$-comodule. Then 
    \begin{align*}
        \Hom_{-B}(M,N)&\cong \Hom^C(M,N\ou BC)\\
        f&\mapsto (f\ou BC)\circ\delta\\
        (N\ou B\varepsilon)\circ F&\mapsfrom F.
    \end{align*}
\end{rem}
While this is of course just part of spiel of the standard pair of adjoint functors between right $B$-modules and right $C$-comodules, the other versions we need are either variants of it tailored to our needs or merely inspired by this and in need of separate proof.

\begin{prop}\label{skew bijection for colinear maps}
Let $\cL$ be a $B$-Hopf algebroid and $N$ a faithfully flat left $\BB$-module, then we have
\begin{align}
    \Hom_{\BB-}(\cL,N)\cong&\,{}^{\cL}\Hom(\cL,\cL\ot_{\BB}N)\label{equ. correspondence 1},\\
    \Hom_{\BB-}(\cL\ot_{\BB}\cL,N)\cong&\,{}^{\cL}\Hom(\cL\ot_{B}\cL,\cL\ot_{\BB}N)\label{equ. correspondence 2}.
\end{align}
Similarly, if $N$ is a faithfully flat left $B$-module, we have
\begin{align}
    \Hom_{B-}(\cL,N)\cong&\Hom{}^{\cL}(\cL,\cL\ot_{B}N)\label{equ. correspondence 3},\\
    \Hom_{B-}(\cL\ot_{B}\cL,N)\cong&\Hom{}^{\cL}(\cL\ot_{ \BB}\cL,\cL\ot_{B}N)\label{equ. correspondence 4}.
\end{align}
\end{prop}
\begin{proof}
    More precisely, the correspondence in (\ref{equ. correspondence 1}) is given by
\[F_{f}(X)=X_{+}\ot f(X_{-})\in \cL\ot_{\BB}N,\]
for any $X\in \cL$ and $f\in \Hom_{\BB-}(\cL,N)$. This map is well-defined because $f$ is left $\BB$-linear, and colinear because of \eqref{equ. inverse lamda 5}.

Conversely, for every $F\in {}^{\cL}\Hom(\cL,\cL\ot_{\BB}N)$, we can define $f_F\colon \CL\to N$ by
\[1\ot f_{F}(X)=X_{[+]}F(X_{[-]})\in\CL\ou\BB N.\]
Namely, the right hand side defines a map to $\CL\ou\BB N$ because $F$ is left $B$-linear. One can then check that the right hand side is of the form $1\ot f_F(X)$ for some $f_F(X)\in N$ using colinearity of $F$ and faithful flatness of $N$. We will treat a more complicated case in detail below. It is not hard to check that $f_F$ is left $\ol B$-linear.
 Moreover
\begin{multline*}
    F_{f_F}(X)=X_+\ot f_F(X_-)
    =X_+(1\ot f_F(X_-))
    \\=X_+(X_{-[+]} F(X_{-[-]}))
    =X\o{}_{+}X\o{}_{-}F(X\t)=F(X).
\end{multline*}Also, we have
\begin{align*}
    f_{F_{f}}(X)=X_{[+]}F_{f}(X_{[-]})=X_{[+]}X_{[-]+}\ot f(X_{[-]-})=X\t{}_{[+]}X\t{}_{[-]}\ot f(X\o)=f(X).
\end{align*}
We turn to the correspondence in (\ref{equ. correspondence 2}).
Given $\phi\in \Hom_{\BB-}(\cL\ot_{\BB}\cL,N)$, we can efine $\Phi_\phi\colon \CL\ou B\CL\to \CL\ou\BB N$ by 
\[\Phi_{\phi}(X,Y)=X_{+}\,Y_{+}\ot \phi(Y_{-}\ot X_{-}).\]
Namely, it is no problem to define 
\begin{align*}\Phi_1\colon \CL\ot\CL\ot\CL&\to\CL\ou\BB N\\
X\ot X'\ot Y&\mapsto XY_+\ot\phi(Y_-\ot X')
\end{align*}
because $\phi$ is left $\BB$-linear. Since 
\begin{multline*}
\Phi_1(X\ol b\ot X'\ot Y)
=X\ol bY_+\ot\phi(Y_-\ot X')
=XY_+\ot \phi(Y_-\ol b\ot X')
\\=XY_+\ot\phi(Y_-\ot\ol bX')
=\Phi_1(X\ot\ol bX'\ot Y),
\end{multline*}
we can define 
\[\Phi_2\colon \CL\ot\CL\ni X\ot Y\mapsto \Phi_1(X_+\ot X_-\ot Y)\in\CL\ou BN,\]
and finally
\begin{multline*}
    \Phi_2(Xb\ot Y)=(Xb)_+Y_+\ot\phi(Y_-\ot  (Xb)_-)
    =X_+bY_+\ot\phi(Y_-\ot  X_-)
    \\=X_+(bY)_+\ot\phi((bY)_-\ot X_-)
    =\Phi_2(X\ot bY)
\end{multline*}
shows that $\Phi_\phi$ is well defined. It is left $B$-linear by \eqref{equ. inverse lamda 9}, and that it is a left comodule map follows from \eqref{equ. inverse lamda 5}.

Conversely, given $\Phi\in {}^{\cL}\Hom(\cL\ot_{B}\cL,\cL\ot_{\BB}N)$, we can define $\phi_\Phi\colon \CL\ou\BB\CL\to N$ by
\[1\ot \phi_{\Phi}(X\ot Y)=X_{[+]}\,Y_{[+]}\,\Phi(Y_{[-]}\ot\,X_{[-]})\in\CL\ou\BB N.\]
In detail, we first define 
\[\phi_1\colon \CL\ot \CL\ot\CL\to \CL\ou \BB N\]
by $\phi(X\ot X'\ot Y)=XY_{[+]}\Phi(Y_{[-]}\ot X')$, using the fact that $\Phi$ is left $B$-linear. Since
\begin{multline*}
    \phi_1(Xb\ot X'\ot Y)
    =XbY_{[+]}\Phi(Y_{[-]}\ot X')
    =XY_{[+]}\Phi(Y_{[-]}b\ot X')
    =XY_{[+]}\Phi(Y_{[-]}\ot bX')
\end{multline*}
we can define 
\[\phi_2\colon \CL\ot\CL\ni X\ot Y\mapsto \phi_1(X_{[+]}\ot X_{[-]}\ot Y)\in \CL\ou\BB N \]
which is easily seen to factor over $\CL\ou\BB\CL$. Since
\begin{align*}
   \delta(\phi_2(X\ot Y))
   =&X_{[+]}\o Y_{[+]}\o\Phi(Y_{[-]}\ot X_{[-]})\mo\ot X_{[+]}\t Y_{[+]}\t\Phi(Y_{[-]}\ot X_{[-]})\z
   \\=&X\o{}_{[+]}Y\o{}_{[+]}\Phi(Y\o{}_{[-]}\ot X\o{}_{[-]})\mo\ot X\t Y\t \Phi(Y\o{}_{[-]}\ot X\o{}_{[-]})\z
   \\=&X\o{}_{[+]}Y\o{}_{[+]}Y\o{}_{[-]}\o X\o{}_{[-]}\o\ot X\t Y\t\Phi(Y\o{}_{[-]}\t\ot X\o{}_{[-]}\t)
   \\=&X\o{}_{[+][+]}Y\o{}_{[+][+]}Y\o{}_{[+][-]}X\o{}_{[+][-]}\ot X\t Y\t \Phi(Y\o{}_{[-]}\ot X\o{}_{[-]})
   \\=&1\ot\phi_2(X\ot Y),
\end{align*}
we have $\phi_2(X\ot Y)\in{{}^{\co\CL}(\CL\ou\BB N)}\cong {}^{\co\CL}\CL\ou\BB N\cong N$ (by faithful flatness of $N$ and thus $\phi_\Phi$ is well-defined. That it is left $\BB$-linear is easy to check.

We omit verifying that the bijections we defined are mutually inverse, i.~e.\ $\Phi_{\phi_\Phi}=\Phi$ and $\phi_{\Phi_\phi}=\phi$

The correspondence in (\ref{equ. correspondence 3}) is similar. Given $f\in \Hom_{B-}(\cL,N)$, we can define 
\[F_{f}(X)=X_{[+]}\ot f(X_{[-]})\in \cL\ot_{B}N.\]
Conversely, given $F\in \Hom{}^{\cL}(\cL,\cL\ot_{B}N)$, we can define $f_F$ by 
\[1\ot f_{F}(X)=X_{+}\,F(X_{-})\in \cL\ot_{B}N.\]
Finally the correspondence in \eqref{equ. correspondence 4} is described by
\[\Phi_\phi(X\ot Y)=X_{[+]}Y_{[+]}\ot\phi(Y_{[-]}\ot X_{[-]})\]
for $\phi\in \Hom_{B-}(\CL\ou B\CL,N)$
and 
\[1\ot \phi_\Phi(X\ot y)=X_+Y_+\Phi(Y_-\ot X_-)\]
for $\Phi\in\Hom^\CL(\CL\ou\BB\CL,\CL\ou BN)$.
\end{proof}

\section{(Anti-)cleft extensions of Hopf algebroids}

In this section, we will study anti-cleft extensions and twisted crossed products of Hopf algebroids. Moreover, we are going to show anti-cleft extensions, twisted crossed products with invertible 2-cocycle and Hopf Galois extensions with normal basis properties are equivalent. Finally, we will also study the equivalence classes of cleft extensions.

\subsection{Skewed convolutions}

 In the theory of Hopf algebroids, convolution of maps is often not the ``right notion'' or not even defined. The prime example is of course the definition of a Hopf algebroid itself, which is \emph{not} by demanding that the identity be convolution invertible (with the antipode being the inverse). The definition of a cleft extension in the Hopf algebra case centers on convolution invertibility of the ``cleaving map''. The most banal cleft extension is the Hopf algebra itself with the identity as the cleaving map and the antipode as its inverse. Visibly our definition has to be different, and we start by proposing a strangely twisted version of convolution invertibility of maps defined on a Hopf algebroid.

\begin{defi}
Let $\CL$ be a Hopf algebroid over $B$, and $A$ an $B^e$-ring. Write $\eta_A\colon B^e\to A$ for the unit of the $B^e$-ring.
For $f\in\Hom_{-\ol B}(\CL,A)$ and $g\in\Hom_{\ol B-}(\CL,A)$ we define the skewed convolution product $f\skpr g\colon\CL\to A$ by
\[f\skpr g(X)=f(X_+)g(X_-).\]
We say that $g$ is a right anti-skewed inverse of $f$ and $f$ a left skewed inverse of $g$ if $f\skpr g=\eta_A\varepsilon$.

If $g\in \Hom_{-B}(\CL,A)$ and $f\in\Hom_{B-}(\CL,A)$ we define the anti-skewed convolution product $g\askpr f$ by
\[g\askpr f(X)=g(X_{[+]})f(X_{[-]})\]
and we say that $g$ is a  left anti-skewed inverse for $f$, and $f$ a right skewed inverse for $g$ if $g\askpr f=\eta_A\ol\varepsilon$.

If $f\in\Hom_{B-\ol B}(\CL,A)$ and $g\in\Hom_{\ol B-B}(\CL,A)$ both products are defined. We will call a left and right (anti)-skewed inverse an (anti-)skewed inverse.
\end{defi}
\begin{rem}
\begin{enumerate}
    \item   Let $f\in\Hom_{-\ol B}(\CL,A)$ and $g\in\Hom_{\ol B-}(\CL,A).$
    \begin{enumerate}
        \item If $f$ is also left $B$-linear then $(f\skpr g)(bX)=b(f\skpr g)(X)$.
        \item If $g$ is also right $B$-linear then $(f\skpr g)(\ol bX)=(f\skpr g)(X)b$.
    \end{enumerate}
    \item Let $g\in \Hom_{-B}(\CL,A)$ and $f\in\Hom_{B-}(\CL,A)$.
    \begin{enumerate}
        \item If $g$ is also left $\ol B$-linear then $(g\askpr f)(\ol bX)=\ol b(g\askpr f)(X)$.
        \item If $f$ is also right $\ol B$-linear then $(g\askpr f)(bX)=(g\askpr f)(X)\ol b$.
    \end{enumerate}
\end{enumerate}
  
\end{rem}\label{rem. classical explaination}
\begin{rem}
 Assume that $\CL$ is an ordinary Hopf algebra ($B=k$) with bijective antipode $S$. Then $f(X_+)g(X_-)=f(X\o)g(S(X\t))$. So skewed convolution is convolution but skewed by inserting the antipode. Further, $f$ admits a right anti-skewed inverse iff $f$ has a right convolution inverse $\ol f$, the anti-skewed inverse being $\ol f S\inv$, so it is the convolution inverse \emph{anti}-skewed by the \emph{inverse of} the antipode.
    
    Thus in the general case, the skewed convolution product is a generalization of the ordinary convolution product of $f$ and $gS$ while neither $S$ nor the ordinary convolution are defined, and the right anti-skewed inverse is a generalization of the convolution inverse composed with $S\inv$.

    Similarly the anti-skewed convolution is given by 
    $g\askpr f(X)=g(X\t)fS\inv(X\o)$ in the case of a Hopf algebra with bijective antipode. In particular $(g\askpr f)S=(gS*f).$ Thus the anti-skewed product (though only composed with $S$) is again a version of the convolution product with one factor composed by $S$, and $f$ has a left anti-skewed inverse iff it has a left convolution inverse $\ol f$, the anti-skewed left inverse being $\ol fS\inv$.

    Note also that for a Hopf algebroid,  the identity is the skewed and anti-skewed inverse of itself.
\end{rem}

Ordinary convolution defines an associative multiplication, which immediately implies that inverses are unique. In our situation something similar can be obtained, based on a sort of mixed associativity of the (anti-)skewed products. For $f,g\colon\CL\to A$ we define $(f*g)(X)=f(X\o)g(X\t)$ and $(f\ol*g)(X)=f(X\t)g(X\o)$. There are obviously conditions for $f$ and $g$ for each of the products to be defined, which we do not spell out.
\begin{lem}
    \begin{enumerate}
        \item For $f\in\Hom_{\ol B-B}(\CL,A)$, $h\in\Hom_{B-}(\CL,A)$ and $g\in\Hom_{-\ol B}(\CL,A)$ we have
        \[g\skpr(f\askpr h)=(g\skpr f)*h\]
        and $f\askpr\eta\varepsilon=f.$
    \item For $f\in\Hom_{B-\ol B }(\CL,A)$, $h\in\Hom_{-B}(\CL,A)$ and $g\in\Hom_{\ol B-}(\CL,A)$ we have
    \[h\askpr(f\skpr g)=(h\askpr f)\ol* g\]
    and $f\skpr\eta\overline{\varepsilon}=f.$
    \end{enumerate}
\end{lem}
\begin{proof}
  Apply the three maps to both sides of \eqref{mixed coass 2} and \eqref{mixed coass 1}, respectively, to get the two mixed associativities. The counit identities follow from \eqref{equ. inverse lamda 7} and \eqref{equ. inverse mu 7}.
\end{proof}
\begin{cor}\label{lem. cancelable}
    \begin{enumerate}
        \item Let  $f\in\Hom_{ B-\ol B}(\CL,A)$. If $f$ has a right inverse for $\skpr$, then it is right cancelable for $\askpr$. If $f$ has a left inverse for $\askpr$, then it is left cancelable for $\skpr.$
        \item Let $f\in\Hom_{\ol B-B}(\CL,A)$. If $f$ has a right inverse for $\askpr$, then it is right cancelable for $\skpr$. If $f$ has a left inverse for $\skpr$, then it is left cancelable for $\askpr$.
    \end{enumerate}
\end{cor}
\begin{proof}
  For (1),  let $h\askpr f=0$, and let $g$ such that $f\skpr g=\eta\varepsilon.$ Then 
    \[h=h\askpr\eta\varepsilon=h\askpr (f\skpr g)=(h\askpr f)\ol* g=0.\]
If $f\skpr g=0$ and $h\askpr f=\eta\ol\varepsilon$, then 
    \[g=(\eta\ol{\varepsilon}) \ol{*} g=(h\askpr f)\ol{*} g=h\askpr (f\skpr g)=0.\]
For (2), if $g\skpr f=0$ and $f\askpr h=\eta\ol\varepsilon$, then 
    \[g=g\skpr\eta\ol\varepsilon=g\skpr(f\askpr h)=(g\skpr f)*h=0.\]    
   If $f\askpr h=0$ and $g\skpr f=\eta \varepsilon$, then
   \begin{align*}
       h=\eta\varepsilon\ast h=(g\skpr f)\ast h=g\skpr (f \askpr h)=0.
   \end{align*}
   \end{proof}

\begin{cor}\label{cor. skew inverse is unique}
  If $g\in\Hom_{\ol B-B}(\CL,A)$ has a left skewed inverse $f$ and a right skewed inverse $f'$ then they are equal. In particular, the skewed inverse is uniquely determined if it exists. The same is true with ``skewed'' replacing ``anti-skewed''. 
\end{cor}
\begin{proof}
    \[f=f\skpr\eta\ol\varepsilon=f\skpr(g\askpr f')=(f\skpr g)*f'=\eta\varepsilon*f'=f'.\]
\end{proof}

\begin{lem}\label{skew inverse and linearity}
    Let $f\in\Hom_{B-\ol B}(\CL,A)$ be the skewed inverse of $g\in\Hom_{\ol B-B}(\CL,A)$. 
    \begin{enumerate}
        \item $f$ is right $B$-linear iff $g$ is left $B$-linear.
        \item $g$ is right $\ol B$-linear iff $f$ is left $\ol B$-linear.
    \end{enumerate} 
\end{lem}
\begin{proof} 
  Denote by $r_b$ resp.\ $\ell_b$ the maps given as right resp.\ left multiplication with $b\in B$ (not specifying on which space, by abuse of notation). Then \cref{equ. inverse mu 10} implies that $(g\ell_b)\askpr f=g\askpr(fr_b).$ On the other hand 
  \begin{align*}
      (\ell_bg)\askpr f(X)=b(g\askpr f)(X)=b\ol{\varepsilon(X)}=\ol{\varepsilon(X)}b=(g\askpr f)(X)b=g\askpr(r_bf).
  \end{align*}
  Thus
  \begin{align*}
      (g\ell_b-\ell_bg)\askpr f=g\askpr(fr_b-r_bf),
  \end{align*}
  so assuming $g$ is left $B$-linear we get that $f$ is right $B$-linear by cancelling $g$, and vice versa by cancelling $f$. The proof of the other version is similar.
  \end{proof}
\begin{lem}\label{lem. skew inverse of also colinear}
    Let $\CH$ be a $B$-Hopf algebroid.
    Let $P$ be a right (resp.\ left) $\CH$-comodule algebra with $B\subseteq P^{co\CH}$ (resp.\ $\BB\subseteq {}^{co\CH}P$), and et $f\in\Hom_{B-\ol B}(\CH,P)$ be the skewed inverse of $g\in\Hom_{\ol B-B}(\CH,P)$. Then $g$ is a comodule map iff $f$ is.
\end{lem}
\begin{proof}
    We treat the right comodule case, the other being similar.  Note that to be a comodule map, $f$, resp.\ $g$, has to be a $\BB$-bimodule map first. We have shown already that these conditions are equivalent. 

    Consider the coaction which is a $B^e$-ring map $P\to P\times_B \CH$ because of our condition on $B$. Thus it is easy to see that $\delta f$ is the skewed inverse of $\delta g$. We have to show that
    \begin{align*}
        \tilde f\colon \CH\ni X&\mapsto f(X\o)\ot X\t\in P\times_B \CH\\\intertext{is also the (at least one-sided) skew inverse of }
        \tilde g\colon \CH\ni X&\mapsto g(X\o)\ot X\t\in P\times_B \CH.
    \end{align*}
    Indeed
    \begin{align*}
        \tilde f\skpr\tilde g(X)&=\tilde f(X_+)\tilde g(X_-)\\
        &=(f(X_+\o)\ot X_+\t)(g(X_{-}\o)\ot X_{-}\t)\\
        &=f(X_+\o)g(X_{-}\o)\ot X_+\t\, X_{-}\t\\
        &=f(X_{++}\o)g(X_{-})\ot X_{++}\t\, X_{+-}\\
         &=f(X_{+})g(X_{-})\ot 1\\
        &=\varepsilon(X)1_{P\times \CH}.
    \end{align*}
\end{proof}

\subsection{Anti-cleft extensions}
In the light of the above preparations we can propose the following defionitions:
\begin{defi}\label{Anti-cleft extensions}
    Let $\CH$ and $\cL$ be  Hopf algebroids. 
    \begin{enumerate}
        \item A  reversible right $\CH$-comodule algebra $P$ is anti-cleft if $B$ is a subalgebra of $P^{\co\CH}$, and there is a colinear and left $B$-linear map  $\tilde\rho\colon\CH\to P$ that admits  an anti-skewed inverse.
        \item A  reversible left $\CL$-comodule algebra $P$ is cleft if $\ol B$ is a subalgebra of ${}^{co\CL}P$ and there is a colinear and left $\ol B$-linear map $\tilde\gamma\colon\CL\to P$ that admits a skewed inverse.
    \end{enumerate}
\end{defi}
\begin{rem}
\begin{enumerate}
    \item Under mild hypotheses of flatness of the Hopf algebroid, reversibility of the comodule structure is automatic by \cite{C20}. Requiring it explicitly underlines the idea that being a cleft extension is a property defined by structure maps (the cleaving and the reverse comodule structure) satisfying equations, as opposed to the more general notion of Hopf Galois extension.
    \item The definitions are such that a left (resp.\ right) comodule algebra over $\CH$ is cleft (resp.\ anti-cleft) if and only if it is anti-cleft (resp.\ cleft) when considered as a right (resp.\ left) comodule algebra over the coopposite $\CH^{\cop}$. Indeed,  let $P$ be a right $\CH$-comodule algebra with the coaction $\delta:P\to P\diamond_{B}\CH,\quad p\mapsto p\z\ot p\o$,  the left $\CH^{cop}$-coaction is given by $\delta': P\to \CH^{cop}\diamond_{\BB} P, p\mapsto p\o\ot p\z$. This is well defined as for the original right coaction we have the fact that $\ol{b}  p\z\ot p\o=p\z\ot b p\o$. So the induced left coaction read 
    $b p\o\diamond_{\BB}p\z=p\o\diamond_{\BB}\ol{b}  p\z$, which is well defined. In addition, assume $N\subseteq P$ is a cleft extension and  $\trho$ is the corresponding map with its anti-skew inverse $\rho$. Then it is clear that both $\trho$ and $\rho$ is left $\CH^{cop}$-colinear and $\rho$ is the skew inverse of $\trho$.
    \item 
    As we have seen, we could equally well have demanded the existence of a colinear and left $B$-linear $\rho\colon \CH\to P$ admitting a skewed inverse. The maps $\rho$ and $\tilde\rho$ determine each other and we will refer to the pair $(\rho,\tilde\rho)$ as a \textit{cleaving}.
    \item If $(\rho,\tilde\rho)$ is a cleaving, then $\rho(1)$ is invertible with inverse $\tilde\rho(1)$, and both are contained in $P^{\co\CH}$. In particular both elements commute with $\ol B\subset P$. Therefore a new cleaving $(\rho',\tilde\rho')$ can be defined by $\tilde\rho'(X)=\rho(1)\tilde\rho(X)$ and $\rho'(X)=\rho(X)\tilde\rho(1)$. Since this new cleaving is \emph{unital} in the sense that both maps preserve units, we shall always assume in the sequel that our cleavings are unital.
\end{enumerate}

\end{rem}

The following proposition will be useful later.
\begin{prop}\label{prop. cleaving maps}
    Given an anti $\CH$-cleft extension $N\subseteq P$ with cleaving  $(\rho,\trho)$ we define 
    \begin{align*}
        \CH\ot N\to N: X\ot n\mapsto X\wedge n:=\trho(X_{+})\,n\,\trho(X_{-})
    \end{align*}
and
\begin{align*}
        \CH\ot N\to N: X\ot n\mapsto X\vee  n:=\rho(X_{+})\,n\,\rho(X_{-}).
    \end{align*}
    Then we have
    \[X_{+}\wedge(X_{-}\vee n)=\varepsilon(X)n,\quad X_{[+]}\vee(X_{[-]}\wedge n)=n\varepsilon(X).\]
Also, we have
\[\trho(X\o{}_{+})\ot_{\BB}\rho(X\o{}_{-})\trho(X\t{}_{+})\ot_{\BB}\rho(X\t{}_{-})=\trho(X{}_{+})\ot_{\BB}1\ot_{\BB}\rho(X{}_{-})\]
\end{prop}
\begin{proof}It is not hard to see both the actions are well defined as both $\rho$ and $\trho$ are $\BB$-bimodule maps and $N$ commutes with $\BB$ as (subalgebras in $P$). Indeed, $\delta(n\,\overline{b})=\delta(\overline{b}\,n)$ for any $n\in N$ and $\overline{b}\in\BB$. We have $X\wedge n,X\vee n\in N=P^{\co\CH}$ because $\rho$ and $\trho$ are right colinear.

    For the first equation, we have
    \begin{align*}
       X_{+}\wedge(X_{-}\vee n)
       =&\trho(X_{++})\rho(X_{-+})\,n\,\rho(X_{--})\trho(X_{+-})\\
       =&\trho(X_{+})\rho(X{}_{-}\o{}_{+})\,n\,\rho(X{}_{-}\o{}_{-})\trho(X_{-}\t)\\
       =&\trho(X_{+})\rho(X{}_{-}{}_{+})\,n\,\rho(X{}_{--[+]})\trho(X{}_{--[-]}) \\
       =&\trho(X_{+})\rho(X{}_{-}{}_{+})\,n\,\overline{\varepsilon(X{}_{--})}\\
       =&\trho(X_{+})\rho(X{}_{-})\,n\\
       =&\varepsilon(X)n.
    \end{align*}
    The second equation is similar. For the last equality, we have
      \begin{align*}
        \trho&(X\o{}_{+})\ot\rho(X\o{}_{-})\trho(X\t{}_{+})\ot\rho(X\t{}_{-})\\
        =&\trho(X{}_{+})\ot\rho(X{}_{-[+]})\trho(X{}_{-[-]+})\ot\rho(X{}_{-[-]-})\\
        =&\trho(X{}_{+})\ot\rho(X{}_{-}\t{}_{[+]})\trho(X{}_{-}\t{}_{[-]})\ot\rho(X{}_{-}\o)\\
        =&\trho(X{}_{+})\ot\overline{\varepsilon(X_{-}\t)}\ot\rho(X{}_{-}\o)\\
        =&\trho(X{}_{+})\ot 1\ot\rho(X{}_{-}).
    \end{align*}
\end{proof}
\begin{rem}
    $E:=\End(N)$ is an $R^e$-ring with $R$ mapped to left multiplication and $\ol R$ mapped to right multiplication. If we define $f^\wedge,f^\vee\colon\CH\to \End(N)$ by $f^\wedge(X)(n)=X\wedge n$ and $f^\vee(X)(n)=X\vee n$ then the two equations above read $f^\wedge\skpr f^\vee=\eta_E\varepsilon$ and $f^\vee\askpr f^\wedge=\eta_E\ol \varepsilon$ 
\end{rem}
\begin{lem}\label{def. anti cleft extension}
    Let $\CH$ be a Hopf algebroid over $B$, $P$ a right $\CH$-comodule algebra with $B\subseteq   N=P{}^{co\CH}$. Assume that  $P$ is a faithfully flat left $B$-module. Then the following are equivalent:
    \begin{itemize}
        \item [(1)] There is an unital right $\CH$-comodule map $\rho:  \CH\to P$, such that $\rho$ is  right $B$-linear, and  $i: \CH\ot_{B}P\to P\diamond_{B}\CH$ given by $i(X\ot_{B} p)= \rho(X\o)p\di X\t$ is bijective.
 \item[(2)]  There is an unital right $\CH$-comodule map $\trho:  \CH\to P$, such that $\trho$ is left $B$-linear, and  $\ti: P\diamond_{B}\CH\to\CH\ot_{B}P$ given by $\ti(p\ot X)= X_{[+]}\ot \trho(X_{[-]})\,p$ is bijective.
 \item[(3)] $N\subseteq P$ is an anti $\CH$-cleft extension.
    \end{itemize}
\end{lem}
\begin{proof}
   
    We can check that $i$ and (and $\ti$, respectively) is well defined as  $\rho$ is right $B$-linear (and $\trho$ is left $B$-linear, respectively). 
    
    First, we show $(1)\Rightarrow(3)$. We define $F:=i^{-1}|_{1_{P}\di\CH}:\CH\to \CH\ot_{B}P, X\mapsto  X^{\delta}\ot_{B}X^{\gamma}$. We can see that \begin{align}\label{equ.anti cleaving map 1}
\rho(X^{\delta}\o)X^{\gamma}\di X^{\delta}\t=&1\di X,\\
\label{equ. anti cleaving map 2}
X\t{}^{\delta}\ot_{B} X\t{}^{\gamma}\rho(X\o)=&X\ot_{B}1,\\
\label{equ. anti cleaving map 5}
(a\Bar{a'}X b\Bar{b'})^{\delta}\ot_{B}(a\Bar{a'}X b\Bar{b'})^{\gamma}=&\Bar{a'} X^{\delta}\Bar{b'}\ot_{B}\Bar{b} X^{\gamma}\Bar{a},
\end{align}
for any $b\in B$, $X\in \CH$. As $F$ is right $\CH$-colinear,  by Proposition \ref{skew bijection for colinear maps}, we can define a left $B$-linear map $\trho:\CH\to P$ by $\trho(X)=X_{+}\,F(X_{-})=X_{+}X_{-}{}^{\delta}\ot X_{-}{}^{\gamma}$. As a result, $F(X)=X_{[+]}\ot \trho(X_{[-]})$ and $i^{-1}(p\ot X)=X_{[+]}\ot \trho(X_{[-]})p$. It follows from (\ref{equ. anti cleaving map 5}) that $\trho$ is a $\BB$-bimodule map. Therefore, (\ref{equ.anti cleaving map 1})  can be written as $\rho(X_{[+]}\o)\, \trho(X_{[-]})\ot X_{[+]}\t=1\ot X$. By applying $(\id\ot \varepsilon)$ on both sides, we get $\rho(X_{[+]})\trho(X_{[-]})=\overline{\varepsilon(X)}$. Similarly, (\ref{equ. anti cleaving map 2}) can be written as $X{}_{[+]}\ot\trho(X{}_{[-]+})\rho(X{}_{[-]-})=X\t{}_{[+]}\ot\trho(X\t{}_{[-]})\rho(X\o)=X\ot 1$. By Proposition \ref{skew bijection for colinear maps}, $\trho(X{}_{+})\rho(X{}_{-})=X_{+}X{}_{-[+]}\ot\trho(X{}_{-[-]+})\rho(X{}_{-[-]-})=X_{+}X_{-}=\varepsilon(X)$. Also, by Lemma \ref{lem. skew inverse of also colinear}, $\trho$ is right $\CH$-colinear.

Second, we show $(3)\Rightarrow (2)$. Given cleaving map $(\rho, \trho)$, we can define $\ti^{-1}(X\ot p):=\rho(X\o)p\di X\t$. Indeed, on the one hand,
\begin{align*}
    \ti\circ \ti^{-1}(X\ot p)=&\ti(\rho(X\o)p\di X\t)=X\t{}_{[+]}\ot \trho(X\t{}_{[-]})\rho(X\o)p\\
    =&X{}_{[+]}\ot \trho(X{}_{[-]+})\rho(X{}_{[-]-})p=X\ot p.
\end{align*}
On the other hand,
\begin{align*}
    \ti^{-1}\circ \ti(p\ot X)=&\ti^{-1}(X_{[+]}\ot \trho(X_{[-]})p)= \rho(X_{[+]}\o)\trho(X_{[-]})p\ot X_{[+]}\t\\
    =&\rho(X\o{}_{[+]})\trho(X\o{}_{[-]})p\ot X\t=\overline{\varepsilon(X\o)}p\ot X\t\\
    =&p\ot X.
\end{align*}

Finally, we show $(2)\Rightarrow (1)$. We observe that 
\begin{align*}
    \Hom^{\CH}(\CH, P\di \CH)&\cong \Hom_{\BB}(\CH, P)\\
    \Phi&\mapsto (\id\ot\varepsilon)\circ \Phi\\
    (\phi\ot \id)\circ\Delta&\mapsfrom \phi
\end{align*}
Define $F':=\ti^{-1}|_{\CH\ot_{B}1_{P}}:\CH\to P\di\CH$, which is right $\CH$-colinear. Then we can define a left $\BB$-linear map  $\rho:=(\id\ot \varepsilon)\circ F'$. Hence $F'(X)=\rho(X\o)\ot X\t$ and $\ti^{-1}(X\ot p)=\rho(X\o)p\ot X\t$. By the same method as above we can see $F'(Xb)=F'(X)b$ (where the right $B$-module of $P\di\CH$ is given by the right $B$-module on $P$), which results in $\rho$ is right $B$-linear. Moreover, as $\ti$ is invertible, we can also get $\rho(X_{[+]})\trho(X_{[-]})=\overline{\varepsilon(X)}$ and $\trho(X_{+})\rho(X_{-})=\varepsilon(X)$. Also, by Lemma \ref{lem. skew inverse of also colinear}, $\rho$ is right $\CH$-colinear.
\end{proof}

\subsection{Three characterizations of anti-cleft extensions}

The following is  a generalization of a result of Doi \cite{Doi}:
\begin{lem}\label{Doi-integrals}
Let $P$ be a  reversible right $\CH$-comodule algebra and $N=P^{\co\CH}$.
\begin{enumerate}
    \item If $\rho\colon \CH\to P$ is right colinear and unital  then the equalizer
    \[N\to P\rightrightarrows P\di\CH\]
    defining the coinvariants is split by left $N$-module maps; in particular 
    $\pi=\pi_\rho\colon P\to N, \quad p\mapsto p\rz\rho(p\rmo)$ is a left $N$-module map splitting the inclusion $N\subset P$. In particular, for every right $N$-module $M$ the unit $M\to (M\ou NP)^{\co\CH}$ of the adjunction in \Cref{hopf module adjunction} is an isomorphism.
    \end{enumerate}
    Now assume that $P$ is anti-right Galois. Then
    \begin{enumerate}[resume]
    \item If $\pi\colon P\to N$ is a left $N$-module map splitting the inclusion $N\subset P$, then 
    $\rho=\rho_\pi\colon\CH\to P;\quad X\mapsto X\yi{}\pi(X\er{})$ is right $\CH$-colinear and unital. 
    \item If $\rho\colon\CH\to P$ is right colinear and unital, then $\rho_{\pi_\rho}=\rho$.
    \item For a left linear splitting $\pi\colon P\to N$ we have $\pi_{\rho_\pi}=\pi$
\end{enumerate}
\end{lem}
\begin{proof}
For (1),    $\pi$ is well defined as a map to $P$ because $\rho$ is left $\ol B $-linear. Furthermore
\begin{align*}
\delta(p\rz\rho(p\rmo))=&p\rz\z\rho(p\rmo)\z\ot p\rz\o\rho(p\rmo)\o\\
    =&p\rz\z\rho(p\rmo\o)\ot p\rz\o p\rmo\t\\
    =&p\rz\rz\z\rho(p\rmo)\ot p\rz\rz\o p\rz\rmo\\
    =&p\rz\rho(p\rmo)\ot 1.
\end{align*}
So $\pi$  has its image in $N$; clearly it is left $N$-linear and satisfies $\pi(n)=n$ for all $n\in N$. One can check that the map $\phi\colon P\di\CH\to P$ defined by $\phi(p\ot X)=p\rz\rho(p\rmo X)$ together with $\pi$ splits the equalizer. 

For (2), we assume that $P$ is right anti-Galois. Given $\pi\colon P\to N$ left $N$-linear and unital, $\rho_\pi$ is obviously unital, and it is right colinear because of \cref{equ. anti translation map 1}.

For (3), we have $\rho_{\pi_\rho}(X)=X\yi{}{X\er{}}\rz\rho({X\er{}}\rmo)=\yi{X\t}\er{X\t}\rho( X\o)=\ol{\varepsilon(X\t)}\rho(X\o)=\rho(X)$
using \eqref{equ. regular anti translation map 2}.

For (4), we  have $\pi_{\rho_\pi}(p)=p\rz\yi{p\rmo}\pi(\er{p\rmo})=\pi(p)$ by \eqref{equ. regular anti translation map 1}. 
\end{proof}

\begin{defi}\label{def. normal basis}
  A right comodule algebra $P$ over a $B$-Hopf algebroid $\CH$ has the normal basis property if $B$ is a subalgebra of $N=P^{\co\CH}$ and $P\cong N\ot_B\CH$ as right $\CH$-comodule and left $N$-module.
\end{defi}
We remark that under the assumption that $\CH$ is faithfully flat in particular as a left $B$-module, a comodule algebra with a normal basis is left faithfully flat over its coinvariant subalgebra.

\begin{defi}\cite{BB}\label{defi. 2 cocycle and twisted module algebra}
    Let $\CH$ be a left bialgebroid over $B$ and $N$ be a $B$-ring. We say $\CH$ measures $N$, if there is a linear map (called measuring) $\la: \CH\ot{} N\to N$, $X\ot n\mapsto X\la n$, such that 
    \begin{align}
(b \overline{b'}X)\la n&=b(X\la n)b', \label{2-cocycle 1}\\
X\la 1_{N}&=\varepsilon(X), \label{2-cocycle 2}\\
X\la (mn)&=(X\o\la m)(X\t \la n), \label{2-cocycle 3}
\intertext{
    for any $m,n\in M$, $b, b'\in B$ and $X\in \CH$.
    Given such a measuring $\la$, a map $\sigma: \CH\ot_{\BB}\CH\to N$ is called a 2-cocycle of $\CH$ over $N$, if it satisfies:}
\sigma(b\overline{b'}X,Y)&=b\sigma(X,Y)b',
\label{2-cocycle 4}\\
\sigma(X\ot_{\BB}1)&=\sigma(1\ot_{\BB}X)=\varepsilon(X),
\label{2-cocycle 5}\\
(X\o\la b)\,\sigma(X\t,Y)&=\sigma(X,b\,Y),
\label{2-cocycle 6}\\
X\o\la\sigma(Y\o,Z\o)\sigma(X\t,Y\t Z\t)
&=\sigma(X\o,Y\o)\sigma(X\t Y\t,Z),
\label{2-cocycle 7}\\
\sigma(X\o,Y\o)(X\t\,Y\t\la n)
&=(X\o\la(Y\o\la n))\sigma(X\t,Y\t)
\quad \textup{\&}\quad 1\la n=n,
\label{2-cocycle 8}
\intertext{
        for any $X, Y, Z\in \CH$.  We call $\sigma$ invertible if there is a map $\sigma^{-1}:\CH\ot_{B}\CH\to N$ such that }
\sigma^{-1}(b\overline{b'}X,Y)
&=b\,\sigma^{-1}(X,Y)\,b',
\label{2-cocycle 9}\\
\sigma^{-1}(X\o,Y)(X\t\la b)
&=\sigma^{-1}(X,\overline{b}Y),
\label{2-cocycle 10}\\
\sigma(X\o,Y\o)\sigma^{-1}(X\t,Y\t)
&=X\la(Y\la 1),
\label{2-cocycle 11}\\
\sigma^{-1}(X\o,Y\o)\sigma(X\t,Y\t)
&=(XY)\la 1.
\label{2-cocycle 12}
\end{align}    
\end{defi}

\begin{rem}
By \cite[Lemma 4.10]{BB}   the inverse $\sigma^{-1}$ satisfies
\[X\la \sigma^{-1}(Y, Z)=\sigma(X\o, Y\o\,Z\o)\sigma^{-1}(X\t\,Y\t, Z\t)\sigma^{-1}(X\th, Y\th).\]
We can derive that $\sigma^{-1}$ satisfies the `right' 2-cocycle condition:
\begin{align}\label{equ. right cocycle}
  \sigma^{-1}(X\o, Y\o Z\o)(X\t\la \sigma^{-1}(Y\t, Z\t))=\sigma^{-1}(X\o Y\o, Z)\sigma^{-1}(X\t, Y\t).
\end{align}
It is not hard to see the above equality is well defined.
We have
\begin{align*}
    \sigma^{-1}&(X\o, Y\o Z\o)(X\t\la \sigma^{-1}(Y\t, Z\t))\\
    =&\sigma^{-1}(X\o, Y\o Z\o)\sigma(X\t, Y\t\,Z\t)\sigma^{-1}(X\th\,Y\th, Z\th)\sigma^{-1}(X\fo, Y\fo)\\
    =&\varepsilon(X\o Y\o Z\o)\sigma^{-1}(X\t\,Y\t, Z\t)\sigma^{-1}(X\th, Y\th)\\
    =&\sigma^{-1}(X\o Y\o, Z)\sigma^{-1}(X\t, Y\t).
\end{align*}
    
\end{rem}

\begin{LD}\cite{BB}\label{Lemma. crossed product}
    \begin{enumerate}
        \item 
    Let $(\sigma,\la)$ be a map satisfies (\ref{2-cocycle 1})-(\ref{2-cocycle 6}). Then $N\ot_{B}\CH$ is an associative algebra with product
    \[(n\ot X)(m\ot Y)=n\,(X\o\la m)\, \sigma(X\t, Y\o)\ot X\th\, Y\t\]
    and unit $1\ot 1$
    if and only if $(\sigma,\la)$ satisfy (\ref{2-cocycle 7})-(\ref{2-cocycle 8}).
    We denote it by  $N\#_{\sigma}\CH$ and call it a crossed product. The crossed product is a $\CH$-comodule algebra with a normal basis.
    \item Conversely, every $\CH$-comodule algebra with a normal basis is a crossed product. 
    \end{enumerate}
\end{LD}

For any right $\CH$-comodule algebra $P$ and algebra map $N\to P^{\co\CH}$ there is a weak monoidal functor $\mathcal F\colon\LComod\CH\to\BiMod N$ defined by $\mathcal F(V)=P\Box^\CH V$; the weak monoidal structure is given by the product in $P$:
\[\xi\colon (P\Box^\CH V)\ou N(P\Box^\CH W)\to P\Box^\CH(V\ou BW);\quad p\ot v\ot q\ot w\mapsto pq\ot v\ot w.\]
In particular we showed in \cite[Thm.~3.31]{HS25} that $\mathcal F$ is a monoidal functor if $P$ is an anti-Galois extension of $N$, using certain faithful flatness assumptions.
\begin{prop}\label{monoidal functor}
    Let $P=N\#_\sigma\CH$ be the crossed product associated to a cocycle $(\la,\sigma)$.
    \begin{enumerate}
        \item $P\Box^\CH V\cong N\ou BV\in\BiMod N$ for any $V\in\LComod\CH$, where the left $N$-module structure on the right hand side is the obvious one, and the right $N$-module structure is given by $(n\ot v)m=n(v\mo\la m)\ot v\z$.
        \item The weak monoidal structure of the functor $\mathcal F=P\Box^\CH\text{---}$ corresponds to 
        \begin{align*}
        (N\ou BV)\ou N(N\ou BW)&\cong (N\ou BV)\ou BW\to N\ou B (V\ou BW) \\n\ot v\ot 1\ot w&\mapsto n\sigma(v\mo,w\mo)\ot v\z\ot w\z.
        \end{align*}
        \item The following are equivalent:
        \begin{enumerate}
            \item The functor $\LComod \CH\ni V\mapsto P\Box^\CH V\in\BiMod N$ is monoidal.
            \item The cocycle $\sigma$ is invertible.
            \item $N\#_\sigma\CH$ is a right anti-Galois extension of $N$.
        \end{enumerate}
        \item If $\sigma$ is invertible, then either one of \eqref{2-cocycle 11} or \eqref{2-cocycle 12} characterizes $\sigma\inv$ if \eqref{2-cocycle 9} and \eqref{2-cocycle 10} are satisfied. 
     \end{enumerate}
\end{prop}
\begin{proof}
    The isomorphism in (1) is
    \begin{align*}
        (N\ou B\CH)\Box^\CH V&\cong N\ou BV\\
        n\ot X\ot v&\mapsto n\varepsilon(X)\ot v\\
        n\ot v\mo\ot v\z&\mapsfrom n\ot v.
    \end{align*}
    It is straightforward to check the claimed $N$-linearity. The claim in (2) is also straightforward. 

    Thus (3)(a) is equivalent to
        \[\xi_{VW}\colon (N\ou BV)\ou BW\to N\ot_{B} (V\ou BW); n\ot v\ot w\mapsto n\sigma(v\mo,w\mo)\ot v\z\ot w\z\]
        being an isomorphism for all $V,W\in\LComod\CH.$

    Assume $\alpha\colon \CH\ou B\CH\to B$ satisfies the analog of \eqref{2-cocycle 9} and \eqref{2-cocycle 10}. Then we can define 
    \begin{align*}
      \zeta_{VW}\colon N\ou B(V\ou BW)&\to (N\ou BV)\ou B W\\
      n\ot v\ot w&\mapsto n\alpha(v\mo,w\mo)\ot v\z\ot w\z.
    \end{align*}
    Indeed it is no problem to define 
    \begin{align*}
        \beta\colon N\ou BV\ou B\CH\di W&\mapsto (N\ou BV)\ou BW\\
        n\ot v\ot Y\ot w&\mapsto n\alpha(v\mo,Y)\ot v\z\ot w
    \end{align*}
    using \eqref{2-cocycle 9}, and then we have
    \begin{multline*}
        \beta(n\ot v\ot \ol bY\ot w)=n\alpha(v\mo,\ol bY)\ot v\z\ot w
        =n\alpha(v\mt,Y)(v\mo\la b)\ot v\ot w
        \\=(n\alpha(v\mo,Y)\ot v\z)b\ot w
        =n\alpha(v\mo,Y)\ot v\z\ot bw
    \end{multline*}
    using \eqref{2-cocycle 10}, proving that $\zeta$ is well-defined. We find that
    \begin{multline*}
        \zeta_{VW}\xi_{VW}(n\ot v\ot w)=\zeta(n\sigma(v\mo,w\mo)\ot v\z\ot w\z)
        \\=n\sigma(v\mt,w\mt)\alpha(v\mo,w\mo)\ot v\z\ot w\z
    \end{multline*}
    and $n(v\mo\la w\mo\la 1)\ot v\z\ot w\z=(n\ot v)\varepsilon(w\mo)\ot w\z=n\ot v\ot w\in (N\ot_{B}V)\ot_{B}W$. Thus $\zeta_{VW}\xi_{VW}=\id$ if
    \begin{equation*}
          \sigma(X\o,Y\o)\alpha(X\t,Y\t)=X\la Y\la 1.
    \end{equation*}
    Also
    \begin{multline*}
        \xi_{VW}\zeta_{VW}(n\ot v\ot w)=\xi_{VW}(n\alpha(v\mo,w\mo)\ot v\z\ot w\z)
        \\=n\alpha(v\mt,w\mt)\sigma(v\mo,w\mo)\ot v\z\ot w\z
    \end{multline*}
    and $n(v\mo w\mo\la 1)\ot v\z\ot w\z=n\varepsilon(v\mo w\mo)\ot v\z\ot w\z=n\ot v\ot w\in N\ot_{B}(V\ot_{B}W)$. Thus $\xi_{VW}\zeta_{VW}=\id$ if 
    \begin{equation*}      \alpha(X\o,Y\o)\sigma(X\t,Y\t)=XY\la 1.
    \end{equation*}
    Thus (3)(b) implies (3)(a).

    Consider the map
    \begin{align*}
        F&=\bigl((N\ot_B\CH)\ot_B\CH\cong (N\ot_B\CH)\ot_N(N\ot_B\CH)\xrightarrow{\rcan}N\ot_B\CH\diamond_B\CH\bigr),
    \end{align*}
    involving the canonical map $\rcan$ defining the right anti-Galois property for $N\#_\sigma\CH.$ 
    We have 
    \begin{align*}
    F(n\ot X\ot Y)&=\rcan(n\ot X\ot 1\ot Y)\\
    &=(n\ot X)\z(1\ot Y)\ot (n\ot X)\o\\
    &=(n\ot X\o)(1\ot Y)\ot X\t\\
    &=n\sigma(X\o,Y\o)\ot X\t Y\t\ot X\th\\
    &=n\sigma(X\o,Y\o)\ot \mu(X\t\ot Y\t).
    \end{align*}
    In particular, $N\#_{\sigma}\CH$ is anti-right Galois if and only if the map 
    \[\xi_{\CH\CH}\colon (N\ou B\CH)\ou B\CH\ni n\ot X\ot Y\mapsto n\sigma(X\o,Y\o)\ot X\t\ot Y\t\in N\ou B\CH\ou B\CH\]
    is bijective. In particular (3)(a) implies (3)(c). 
    
    It remains to show that bijectivity of $\xi_{\CH\CH}$ implies (3)(b). By applying Remark \ref{free comodule bijection} twice, we first show that
 \begin{align}\label{equ. g inv}
     (\xi_{\CH\CH})\inv(n\ot X\ot Y)=n\alpha(X\o,Y\o)\ot X\t\ot Y\t
 \end{align}
 for a suitable map $\alpha\colon \CH\ou B\CH\to B$ satisfying the analogs of \eqref{2-cocycle 9} and \eqref{2-cocycle 10}.
 
     First,  as $(\xi_{\CH\CH})^{-1}\in \Hom^{\CH}(N\ot_B\CH\ot_B\CH^{\bullet}, (N\ot_B\CH)\ot_B\CH^{\bullet})$, there is a map $\Sigma\in \Hom_{-B}(N\ot_B\CH\ot_B\CH, N\ot_B\CH)$, such that \begin{align*}
        (\xi_{\CH\CH})^{-1}(n\ot X\ot Y)=\Sigma(n\ot X\ot Y\o)\ot Y\t,\,\quad \Sigma=(\id\ot \varepsilon)\circ (\xi_{\CH\CH})^{-1}.
    \end{align*}
Moreover, because $\xi_{\CH\CH}\in \Hom^{\CH}(N\ot_B\CH^{\bullet}\ot_B\CH, (N\ot_B\CH^{\bullet})\ot_B\CH)$ and thus $(\xi_{\CH\CH})\inv$ is also $\CH$-colinear with those structures,  we have $\Sigma=(\id\ot \varepsilon)\circ (\xi_{\CH\CH})^{-1}\in \Hom^{\CH}(N\ot_B\CH^{\bullet}\ot_B\CH, N\ot_B\CH^{\bullet})$ since $\id\ot \varepsilon$ preserves the comodule structure in the middle term. So there is $\tilde{\sigma}\in \Hom_{-B}(N\ot_{B}\CH\ot_B\CH, N)$ (notice that the right $B$-linearity of $\tilde{\sigma}$ is given by the left $\BB$-linearity on the second term), such that
\begin{align*}
    \Sigma(n\ot X\ot Y)=\tilde{\sigma}(n\ot X\o\ot Y)\ot X\t,\quad \tilde{\sigma}=(\id\ot \varepsilon)\circ \Sigma
\end{align*}
Since $(\xi_{\CH\CH})^{\pm 1}$ are left $N$-linear, so are $\Sigma$ and $\tilde{\sigma}$, and thus there is a suitable $\alpha\in \Hom_{B-}(\CH\ot_B\CH, N)$, such that
\begin{align*}
    \tilde{\sigma}(n\ot X\ot Y)=n\alpha(X\ot Y),
\end{align*}
which results in (\ref{equ. g inv}).

    Summarizing the construction,  $\alpha$ is the following composition:
    \begin{align*}
        \alpha=(\CH\ot_{B}\CH\xrightarrow{i} N\ot_B\CH\ot_B\CH\xrightarrow{(\xi_{\CH\CH})^{-1}}(N\ot_B\CH)\ot_B\CH\xrightarrow{\id\ot \varepsilon}N\ot_B\CH\xrightarrow{\id\ot \varepsilon}N).
    \end{align*} 
    
We can see $\alpha$ is left $\BB$-linear since $\tilde{\sigma}$ is. 
To show that $\alpha$ satisfies the analog of \eqref{2-cocycle 10}, we first define $L^{3}_{\overline{b}}$ being the left multiplication of $\overline{b}$ on the third term of $N\ot_{B}\CH\ot_{B}\CH$. Similarly, $R^{2}_{b}$ is the right multiplication of $b$ on the second term. 
\begin{align*}
    \alpha(X, \overline{b}\,Y)=&(\id_{N} \ot_{B} \varepsilon)\circ (\id_{N\ot_{B}\CH}\ot_{B} \varepsilon)\circ (\xi_{\CH\CH})\inv(1\ot X\ot \overline{b}\,Y)\\
    =&(\id_{N} \ot_{B} \varepsilon)\circ (\id_{N\ot_{B}\CH}\ot_{B} \varepsilon)\circ L^{3}_{\overline{b}} \circ (\xi_{\CH\CH})\inv(1\ot X\ot Y)\\
     =&(\id_{N} \ot_{B} \varepsilon)\circ R^{2}_{b} \circ(\id_{N\ot_{B}\cL}\ot_{B} \varepsilon)\circ (\xi_{\CH\CH})\inv(1\ot X\ot Y)\\
     =&(\id_{N} \ot_{B} \varepsilon)\circ R^{2}_{b} \circ\Sigma(1\ot X\ot Y)\\
     =&(\id_{N} \ot_{B} \varepsilon)\circ R^{2}_{b} (\alpha(X\o\ot Y)\ot X\t)\\
     =&(\id_{N} \ot_{B} \varepsilon)(\alpha(X\o\ot Y)\, (X\t\la b)\ot X\th)\\
     =&\alpha(X\o\ot Y)\, (X\t\la b),
\end{align*}
where the 2nd step use the fact $\xi_{\CH\CH}$ is $L^{3}_{\overline{b}}$-linear. 

Thus we can define $\zeta$ as in the first part of the proof, and by construction of $\alpha$ we have
\begin{multline*}
    (\xi_{\CH\CH})\inv(n\ot X\ot Y)=\Sigma(n\ot X\ot Y\o)\ot Y\t
    =\tilde\sigma(n\ot X\o\ot Y\o)\ot X\t\ot Y\t
    \\=n\alpha(X\o,Y\o)\ot X\t\ot Y\t
    =\zeta_{\CH\CH}(n\ot X\ot Y).
\end{multline*}
We have already seen that $\zeta_{\CH\CH}\xi_{\CH\CH}=\id$ implies
\[\sigma(X\o,Y\o)\alpha(X\t,Y\t)\ot X\th\ot Y\th=X\o\la (Y\o\la 1)\ot X\t\ot Y\t,\]
and applying $\varepsilon$ to the last two factors yields the analog of \eqref{2-cocycle 11}. Also we have seen that $\xi_{\CH\CH}\zeta_{\CH\CH}=\id$ implies 
\[\alpha(X\o,Y\o)\sigma(X\t,Y\t)\ot X\th\ot Y\th=(X\o Y\o)\la 1\ot X\t\ot Y\t,\]
and applying $\varepsilon$ yields the analog of \eqref{2-cocycle 12}. Thus $\alpha=\sigma\inv$ is an inverse of $\sigma$ in the sense of the definition.

As for (4), note that for $\alpha$ satisfying the analogs of \eqref{2-cocycle 9} and \eqref{2-cocycle 10} we have shown that one of the invertibility conditions \eqref{2-cocycle 11} or \eqref{2-cocycle 12} alone show that $\zeta_{\CH\CH}$ defined by $\alpha$ is a one-sided inverse of $\xi_{\CH\CH}.$ If $\sigma$ is invertible then $\xi_{\CH\CH}$ is bijective and thus a one-sided inverse is two-sided, and the same calculation as above shows that the other invertibility equation is also satisfied.
\end{proof}

\begin{lem}\label{lem. even half anti cleft extension is Galois}
   Let $\CH$
be a Hopf algebroid over $B$ and $P$ a right $\CH$-comodule algebra with $N=P^{\co\CH}$ and such that $B$ is a subalgebra of $N$. If there is a right colinear and right $B$-linear $\rho\colon\CH\to P$ that admits a skewed right inverse $\tilde\rho$ which is also colinear, then $N\subseteq P$ is an anti-right $\CH$-Galois extension. More precisely, the anti-right translation map is given by
    \begin{align}
        X\yi{}\ot_{N}X\er{}=\rho(X_{[+]})\ot_{N}\trho(X_{[-]}),
    \end{align}
    for any $X\in \CH$.
\end{lem}
\begin{proof}
    On the one hand we have
    \begin{align*}
        X\yi{}\z\,X\er{}\ot  X\yi{}\o=&\rho(X_{[+]}\o)\trho(X_{[-]})\ot X_{[+]}\t
        =\rho(X\o{}_{[+]})\trho(X\o{}_{[-]})\ot X\t\\
        =&\overline{\varepsilon(X\o)}\ot X\t=1\ot X.
    \end{align*}
    
  On the other hand, we have to show
    \[p\o\yi{}\ot p\o\er{} p\z =p\ot 1.\]
    Now we know from \Cref{Doi-integrals} that $(P\ou NP^{\bullet})^{\co\CH}\cong P$  if we endow the source with the comodule structure on the right hand tensor factor. In this case the isomorphism can be given by multiplication. So we only have to show that 
    \[p\o\yi{}\ot p\o\er{} p\z \in (P\ot_{N}P)^{\co\CH}.\]
     Indeed,
    \begin{align*}
        \delta(\rho(p\o{}_{[+]})\ot\trho(p\o{}_{[-]})p\z)
        &=\rho(p\t{}_{[+]})\ot\trho(p\t{}_{[-]}\o)p\z\ot p\t{}_{[-]}\t p\o\\
        &=\rho(p\t{}_{[+][+]})\ot\trho(p\t{}_{[+][-]})p\z\ot p\t{}_{[-]}p\o\\
        &=\rho(p\o{}_{[+]})\ot\trho(p\o{}_{[-]})p\z\ot 1.
    \end{align*}  
\end{proof}

\begin{thm}\label{classical char of cleft extensions}
Let $\CH$ be a $B$-Hopf algebroid, and $P$ a right $\CH$-comodule algebra. The following are equivalent
\begin{enumerate}
    \item \label{equiv.cleft} $P$ is anti-cleft.
    \item \label{equiv.galnb} $P$ is an anti-right $\CH$-Galois extension of $N=P^{\co\CH}$ that admits a normal basis.
    \item \label{equiv.crpinv} $P\cong N\#_\sigma\CH$ is a crossed product with an invertible cocycle.
\end{enumerate}
\end{thm}
\begin{proof}
   We first show that (\ref{equiv.cleft}) implies (\ref{equiv.galnb}).  Let $(\rho, \trho)$ be the cleaving maps. Define $\pi\colon P\to N$ as in \Cref{Doi-integrals} using $\rho$. Since now $\rho$ is also right $B$-linear, we have $\pi(\ol bp)=\pi(p)b$. We define
   
    \[\psi:P\to N\ot_{B}\CH,\qquad p\mapsto \pi(p\z)\ot p\o= p\z\rz\,\rho(p\z\rmo)\ot p\o. \]
    Clearly, $\psi$ is right $\CH$-colinear and left $N$-linear. The inverse is given by $\psi^{-1}(n\ot X)=n\trho(X)$. Indeed, on the one hand,
    \begin{align*}
        \psi\circ \psi^{-1}(n\ot X)=&n\trho(X)\z\rz\,\rho(\trho(X)\z\rmo)\ot \trho(X)\o\\
        =&n\trho(X\o)\rz\,\rho(\trho(X\o)\rmo)\ot X\t\\
        =&n\trho(X\o{}_{+})\,\rho(X\o{}_{-})\ot X\t\\
        =&n\ot X,
    \end{align*}
    where the 3rd step uses that $\trho$ is right $\CH$-colinear. On the other hand,
    \begin{align*}
        \psi^{-1}\circ\psi(p)=&p\rz\,\rho(p\rmo{}_{[+]})\trho (p\rmo{}_{[-]})=p\rz\overline{\varepsilon(p\rmo)}=p.
    \end{align*}
    Moreover, by Lemma \ref{lem. even half anti cleft extension is Galois}, we know $N\subseteq P$ is an anti-right $\CH$-Galois extension.

   For the equivalence of (\ref{equiv.galnb}) and (\ref{equiv.crpinv}), note that by Lemma \ref{Lemma. crossed product} we already know every comodule algebra with the normal basis property is a crossed product. Thus we can apply \Cref{monoidal functor}.

    To finish the proof we will assume (3) and (2) and deduce (1). First of all the comodule structure is  reversible because if $P=N\ou B\CH$, then the map $\psi$ in \Cref{def. regular of comodule} identifies with $N\ou B\lambda$ where $\lambda$ is from \Cref{defHopf}. 
    
    We define $\tilde\rho\colon \CH\to P$ by $\tilde\rho(X)=1\#X$, which is clear left $B$-linear and right $\CH$-colinear. Consider
    \[H:\CH\xrightarrow{\hat\tau}P\ou NP\cong (N\ou B\CH)\ou B\CH.\]
    By \eqref{equ. anti translation map 2}, $H$ is colinear if we equip the source with the comodule structure given by $\delta(X)=X_{[+]}\ot X_{[-]}$ and the target with the comodule structure coming from the rightmost tensor factor. By Remark \ref{free comodule bijection}, if we define $\rho:=(\id\ot \varepsilon)\circ H:\CH\to N\ot_{B}\CH=P$ then $H(X)=\rho(X_{[+]})\ot X_{[-]}$  and thus the anti-translation map for $P$ is $\hat\tau(X)=\rho(X_{[+]})\ot\tilde\rho(X_{[-]})$. This implies in particular that $\rho\askpr\tilde \rho(X)=X\yi{} \,X\er{}=\eta\ol\varepsilon(X)$.

There is an obvious left $N$-linear splitting $\pi$ of the inclusion $N\subset P$, namely $\pi(n\ot X)=n\varepsilon(X)$. By definition 
\begin{multline*} 
\rho(X)=(\id\ot \varepsilon)\circ H(X)=(\id\ot \varepsilon)(\rho(X_{[+]})\ot X_{[-]})\\=\rho(X_{[+]})\pi(\tilde\rho(X_{[-]}))=X\yi{}\pi(X\er{})=\rho_{\pi}(X).\end{multline*} Therefore, by  \Cref{Doi-integrals} $\rho$ is right $\CH$-colinear, and we can define $\pi_{\rho}:P\to N$ by $\pi_{\rho}(p)= p\rz\rho(p\rmo)$. We find
    \begin{align*} 
      \pi_\rho(p\z)\trho(p\o)
      &=p\z\rz\rho(p\z\rmo)\trho(p\o)\\
      &=p\rz\rho(p\rmo{}_{[+]})\trho(p\rmo{}_{[-]})\\
      &=p\rz(\rho\askpr\trho)(p\rmo)=p.
    \end{align*} 
    As a result, we have
    \begin{align*}
        (n\ot X_+)\rho(X_-)
        &=(n\ot X_+)\rho(X_{-[+]})\pi(\trho(X_{-[-]}))\\
        &=(n\ot X\o{}_{+})\rho(X\o{}_{-})\pi(\trho(X\t))\\
        &=\pi_\rho(n\ot X\o)\pi(\trho(X\t))\\
        &=\pi(\pi_\rho(n\ot X\o)\,\trho(X\t))\\
        &=\pi(n\ot X)=n\varepsilon(X),
    \end{align*}
    which implies in particular $\trho\skpr\rho=\eta\varepsilon$ and thus $(\rho,\trho)$ is a cleaving.
\end{proof}

\begin{rem}\label{rem. concrete correspondence}
    The proof of \Cref{classical char of cleft extensions} establishes bijections between the data defining a cleft extension, and that defining a crossed product with invertible cocyle; it also describes the translation map of the arising Galois extension and the normal basis. We summarize an explicit description:
    
  Given cleaving map $(\rho,\trho)$, the corresponding normal basis is 
    \begin{align*}
        \psi:P\to& N\ot_{B}\CH\\
        p\mapsto& \pi(p\z)\ot p\o= p\z\rz\,\rho(p\z\rmo)\ot p\o\\
        n\trho(X)\mapsfrom& n\ot X.
    \end{align*}
The translation map is:
\[X\yi{}\ot_{N}X\er{}=\rho(X_{[+]})\ot_{N}\trho(X_{[-]}).\]

Given the normal basis and translation map we can recover the cleaving by
 \[\trho(X)=\psi^{-1}(1\ot X),\qquad \rho(X)=X\yi{}\,X\er{}{}^{l}\,\varepsilon(X\er{}{}^{r}),\qquad \forall X\in \CH\]
where $\psi(X)=:p^l\ot p^r\in N\ou B\CH$.

The cocycle and its inverse can be constructed from a cleaving 
    \begin{align*}
        X\la n:=&\trho(X_{+})n\rho(X_{-}),\\
        \sigma(X, Y):=&\trho(X_{+})\trho(Y_{+})\rho(Y_{-}X_{-}),\\
        \sigma^{-1}(X,Y):=&\trho(X_{+}Y_{+})\rho(Y_{-})\rho(X_{-})&.
    \end{align*}

   Finally if the cocycle in a crossed product is invertible then
    \begin{align*}
        \rho(X)&=\sigma^{-1}(X\o{}_{[+]},X\o{}_{[-]})\ot X\t,\\
        \trho(X)&=1\ot X.
    \end{align*}
    is a cleaving.   
\end{rem}

By using the formula in Remark \ref{rem. concrete correspondence},  there is indeed a one to one correspondence between a stronger anti-cleft extension (namely, require $\trho$ being right $B$-linear or $\rho$ being left $B$-linear) and a stronger normal basis property (namely, require $\psi:P\to N\ot_{B}\CH$ being right $B$-linear).
\begin{lem}\label{rem. bilinear 2-cocycle of L over N}
  In the situation of \Cref{classical char of cleft extensions}, the following are equivalent: 
  \begin{enumerate}
      \item The normal basis $\psi\colon P\to N\ou B\CH$ is right $B$-linear.
      \item $\trho$ is right $B$-linear.
      \item $\rho$ is left $B$-linear.
      \item $X\la b=\varepsilon(Xb)$ for $b\in B$.
      \item $\sigma$ factors over $\CH\ou{B^e}\CH$ and $\sigma^{\pm}(X, Y\ol b)=\sigma^{\pm}(X, Y b)$ for any $X, Y\in\CH$.
  \end{enumerate}
  \end{lem}
\begin{proof}
We can show (1) and (2) are equivalent by using the formula of `(1)$\Rightarrow$(2)' and `(2)$\Rightarrow$(1)' in Remark \ref{rem. concrete correspondence}.  We can show (2) and (3) are equivalent by Lemma~\ref{skew inverse and linearity}. We can show (3) implies (4) by using the formula of `(1)$\Rightarrow$(3)' in Remark \ref{rem. concrete correspondence}. 

Now, we show (4) implies (5).  By using \eqref{2-cocycle 6} and \eqref{2-cocycle 10} of Definition~\ref{defi. 2 cocycle and twisted module algebra} we can show $\sigma^{\pm}$ factor through $\ot_{B^e}$. Also, we have  $\sigma^{\pm}(X,Yb)=\sigma^{\pm}(X,Y\overline{b})$, for any $X\in \CH$ and $b\in B$. Indeed, we have on the one hand,
    \begin{align*}
        \sigma(X\o, Y\o) (X\t\,Y\t \la b)=&\sigma(X\o, Y\o)\varepsilon(X\t \varepsilon(Y\t \overline{b}))=\sigma(X\o, \overline{\varepsilon(Y\t \overline{b})}\,Y\o)\varepsilon(X\t)\\
        =&\sigma(X, Y\overline{b}).
    \end{align*}
    On the other hand,
    \begin{align*}
        (X\o\la( Y\o\la b))\sigma(X\t, Y\t)=&\varepsilon(X\o \overline{\varepsilon(Y\o b)})\sigma(X\t, Y\t)=\varepsilon(X\o )\sigma(X\t, \varepsilon(Y\o b)Y\t)\\
        =&\sigma(X, Yb).
    \end{align*}
We can show (5) implies (1) by using   `(3)$\Rightarrow$(1)' in Remark \ref{rem. concrete correspondence}. Indeed, we have
 \begin{align*}
     \rho(bX)=&\sigma^{-1}(X\o{}_{[+]}, X\o{}_{[-]}\overline{b})\ot X\t=\sigma^{-1}(X\o{}_{[+]}, X\o{}_{[-]}b)\ot X\t\\
     =&\sigma^{-1}(bX\o{}_{[+]}, X\o{}_{[-]})\ot X\t=b\rho(X).
 \end{align*}

\end{proof}

 A left cleft extension is by definition a left $\CL$-comodule algebra $Q$ such that $\ol B\subset M:={^{\co\CL} Q}$, and it is a right anti-cleft extension of $M$ when considered as a right $\CL^\cop$-comodule algebra. In particular it has a normal basis in the sense of an isomorphism $M\ou{\ol B}\CL$ of left $M$-modules and left $\CL$-comodules. It can be written as a crossed product which we will denote
$M\lcp_\alpha\CL^\cop$ where $(\la,\alpha)$ is a cocycle of $\CL^\cop$ over $M$ and multiplication is given by
\[(m\lcp X)(n\lcp Y)=m(X\th\la n)\alpha(X\t,Y\t)\lcp X\o Y\o.\]

\begin{rem}\label{left cross product in terms of right cross product}
If $M=N^\op$ for a $B$-ring $N$ and $(\sigma,\la)$ is an invertible cocycle of $\CL$ over $N$, then we can take $\alpha=\sigma\inv$, which is an invertible cocycle of $\CL^\cop$ over $M$ and obtain a left crosed product as above. In this case we can identify $M\lcp_\alpha\CL^\cop$ with $\CL\diamond_BN$, and write multiplication on the latter in the form 
    \[(X\ot n)(Y\ot m)=X\o\,Y\o\ot \sigma^{-1}(X\t, Y\t)(X\th\la m)\, n.\]  
\end{rem}

\begin{rem}\label{first left cleft remark}
    Let $H$ be a $k$-Hopf algebra with bijective antipode. Then there is an obvious notion of cleftness for a left $H$-comodule algebra, obtained from the notion of a cleft right comodule algebra by reversing all tensor products.
    
    This does not correspond to the above notion. First, in the ordinary Hopf case, a left comodule algebra would be considered cleft if there is a convolution invertible left comodule map $H\to P$. If we specialize our definition to the Hopf case, we get a different condition. Namely, we do get a left comodule map $H\to P$ (because we require a right comodule map $H^\cop\to P$, which is the same thing). However, we require this to be convolution invertible as a map $j\colon H^\cop\to P$, where the classical definition requires it to be convolution invertible on $H$. 
    
    Second, a cleft left comodule algebra $P$ with coinvariants $N$ would have a normal basis which is a left $H$-colinear and \emph{right} $N$-linear isomorphism $P\cong H\ot N$. Multiplication in this case can be written as a left crossed product with multiplication
    \[(g\ot m)(h\ot n)=g\o h\o\ot \kappa(g\t, h\t)(m\triangleleft h\th)n\]
    with a \emph{right} weak action of $H$ on $N$ and suitable cocycle $\kappa$. This formula does not correspond to our formula for the Hopf algebroid case. Also, it does not lend itself to a naive Hopf algebroid generalization due to the presence of a right action; left Hopf algebroids have no good notion of a right action compatible with tensor products, for example no good notion of right module algebra which would give rise to the special case of a smash product (without cocycle).

    In the ordinary Hopf case, it is easy to make the differences disappear. Given an ordinary Hopf cleaving $j\colon H\to P$ for the left comodule algebra $P$, we can compose its inverse with the inverse antipode to obtain a (left) cleaving  in our sense. 
    
    We will come back to a corresponding generalization of the ordinary crossed product formula suitable to the Hopf algebroid situation later.
\end{rem}

\section{The inverse and the Ehresmann Hopf algebroid}

\begin{LD}
   Let $P$ be a right anti-cleft extension of $N$ over $\CH$. Then $P\cong\CH\ot_{B}N$ as right $\CH$-comodule and right $N$-module. We say that $P$ has a weird normal basis.
\end{LD}
\begin{proof}
Let $P$ be a right anti-cleft extension of $N$ over $\CH$. There is an obvious right $N$-linear and $\CH$-colinear map 
\[\tilde\psi\colon\CH\ou BN\to P;\quad X\ot n\mapsto\rho(X)n.\]
The map $\tilde\psi$ is a left module map if we endow the source with the module structure
\begin{align}\label{equ. left N module structure on weird product}
    m(X\ot n)=X_{[+]}\ot(X_{[-]}\la m)n
\end{align}
In fact 
\begin{align*}
    \tilde\psi(X_{[+]}\ot(X_{[-]}\la m)n)
    &=\rho(X_{[+]})(X_{[-]}\la m)n\\
    &=\rho(X_{[+]})\trho(X_{[-]+})m\rho(X_{[-]-})n\\
    &=\rho(X\t{}_{[+]})\trho(X\t{}_{[-]})m\rho(X\o)n\\
    &=m\rho(X)n=m\tilde\psi(X\ot n).
\end{align*}
Combining with the normal basis of $P$ we obtain 
\[\CH\ou BN\xrightarrow{\tilde\psi}P\xrightarrow{\psi\inv} N\#_\sigma \CH\]
with
\begin{align*}
    \psi\inv\tilde\psi(X\ot n)
    &=\psi\inv(\rho(X)n)\\
    &=\pi(\rho(X)\z n)\ot \rho(X)\o\\
    &=\pi(\rho(X\o)n)\ot X\t\\
    &=X\o\vee n\ot X\t
\end{align*}
using
\begin{align*}    \pi(\rho(X)n)&=\rho(X)\rz n\rho(\rho(X)\rmo)\\
    &=\rho(X_+)n\rho(X_-)=X\vee n.
\end{align*}
We show next that $\psi\inv\tilde\psi$ and hence $\tilde\psi$ is an isomorphism: Define
\[\alpha\colon N\ou B\CH\to \CH\ou BN;\quad n\ot X\mapsto X_{[+]}\ot X_{[-]}\wedge n \]
and find
\begin{align*}
    \alpha\psi\inv\tilde\psi(X\ot n)
    &=\alpha(X\o\vee n\ot  X\t)\\
    &=X\t{}_{[+]}\ot X\t{}_{[-]}\wedge (X\o\vee n)\\
    &=X_{[+]}\ot X_{[-]+}\wedge (X_{[-]-}\vee n)\\
    &=X\ot n
\end{align*}
as well as
\begin{align*}
    \psi\inv\tilde\psi\alpha(n\ot X)
    &=\psi\inv\tilde\psi( X_{[+]}\ot X_{[-]}\wedge n)\\
    &=X_{[+]}\o\vee (X_{[-]}\wedge n)\ot X_{[+]}\t\\ 
    &=X\o{}_{[+]}\vee (X\o{}_{[-]}\wedge n)\ot X\t\\
    &=n\ot X.
\end{align*}
\end{proof}

\begin{cor}\label{cor. right ff of weired basis}
  If $\CH$ is faithfully flat as right $B$-module, it follows that a right anti-cleft extension $N\subset P$ is right faithfully flat over $N$.
\end{cor}

\begin{rem}
    There is a formula for the algebra structure of $\CH\ou BN$ induced along $\tilde\psi$, which we could call a weird crossed product. We will find that formula later one. 
\end{rem}
\begin{cor}
    If $P$ is a left cleft comodule algebra over $\CH$, then there is an isomorphism $P\cong \CH\ou{\ol B}N$ of right $N$-modules and left $\CH$-comodules.
\end{cor}
Recall that for another right $\CH$-comodule algebra  $Q$,  reversible as a comodule, we have an isomorphism of algebras 
\[(Q\ou{\ol B}P)^{\co\CH}\cong (P\Box^\CH \ol Q)^\op\]
from \cite[Lemma~3.15]{HS25}
where multiplication on the left hand side is induced by that of $Q\ot P^{\op}$. On the right hand side $\ol Q$ is the opposite algebra of $Q$ endowed with the reverse comodule structure.

A particular case of interest is $Q=\CH$, where the above algebra is a left comodule algebra over $\CH$ and, under suitable flatness assumptions, the inverse of $P$ in the groupoid of bi-Galois extensions. A second important case arises where $Q=P$ and the algebra is the left Ehresmann Hopf algebroid of the Galois extension $P$.
\begin{rem}
    In the sequel we will write $P\inv=(\CH\ou{\ol B}P)^{\co\CH}$.
    In \cite{HS25}, it was shown that $P\inv$ is in fact the inverse of $P$ in the groupoid of bi-Galois extensions; however, the objects of this groupoid are bi-Galois extensions that are faithfully flat both over their right and their left coinvariants on both sides; here we impose less conditions.
    
    Assume that $\CH$ is faithfully flat over $B$ and $\BB$ on both sides. Then  $P$ is faithfully flat over $N$ on both sides as we remarked after defining a normal basis \Cref{def. normal basis} and in \Cref{cor. right ff of weired basis}. With this, the Ehresmann Hopf algebroid $\CL=L(P,\CH)$ is defined, and $P$ is a left $\CL$-Galois extension, but $P$ is not necessarily faithfully flat over $\ol B$ unless we assume this for $N$ over $B$. Therefore the proof in \cite{HS25} that $P\Box^\CH P\inv\cong\CL$ does not apply here. 
    
    Nonetheless, if $P\cong N\#_\sigma\CH$ is right anti-cleft then for any right $\CH$-comodule algebra $Q$ we have
    \[Q\Box^\CH P\inv=Q\Box^\CH(\CH\ou{\ol B}P)^{\co\CH}\cong (Q\Box^\CH (\CH\ou{\ol B}P))^{\co\CH}\cong(Q\ou{\ol B}P)^{\co\CH}\]
    using for the first isomorphism that the equalizer defining $(\CH\ou{\ol B}P)^{\co\CH}=P\Box^\CH\ol\CH$ is split. 
    That is to say, the map
     \[(N\ot_{B}\CH)\di \CH\di \ol \CH\to (N\ot_{B}\CH)\di \ol \CH,\quad n\ot X\ot Y\ot Z\mapsto n\ot \varepsilon(X)Y\ot Z.\]
    contracts (in the sense dual to \cite[p.~70, Lem.~3]{Pareigis1970}) the parallel pair of arrows defining the equalizer $(N\ou B\CH)\Box^\CH\ol\CH$.
    In particular,  if $\CH$ is faithfully flat as indicated above, then
    \[P\Box^\CH P\inv
    \cong(P\ou{\ol B}P)^{\co\CH}\cong \CL.\]
    In \cite{HS25} it was also shown that $P\inv$ is inverse to $P$ on the left, namely 
  \begin{equation}\label{inverse extension iso}
    \CH\to P\inv\Box^{\CL}P;\quad X\mapsto X_+\ot X_-{}\yi{}\ot X_-{}\er{}
    \end{equation}
    is an isomorphism, provided $P$ is faithfully flat over $\ol B$ on the left. Without this assumption, we still have a well-define map 
    \begin{equation}
        \label{inverse extension morph}
        \CH\to P\inv\times_N P
    \end{equation}
    given by the same formula.
\end{rem}

\begin{lem}\label{calculating inverse part one}
    Let $(\la,\sigma)$ be an invertible 2-cocycle of $\CH$ over $N$ and $P\cong N\#_\sigma\CH$ the associated crossed product. Let $Q$ be a  reversible right $\CH$-comodule.
    \begin{enumerate}
        \item We have an isomorphism 
        \begin{align}\label{ehresmann precursor}
    \theta\colon \ol N\ou{\ol B} Q\ni \overline{n}\ot q\mapsto q\rz\ot \rho(q\rmo)n\in (Q\ou {\ol B}P)^{\co\CH}. 
    \end{align} of $\ol N$-bimodules. The bimodule structure on the target comes from the $N$-bimodule structure of $P$. The left $\ol N$-module structure of the source is the obvious one, and the right $\ol N$-module structure is $(\ol n\ot q)\ol m=\ol{(q\o\la m)n}\ot q\z.$

        \item $P\inv$ is a left cleft extension of $N^\op$. The weak action of $\CH^\cop$ on $N^\op$ is the same as the action of $\CH$ on $N$, and the cocycle is the inverse $\sigma\inv.$ We have an isomorphism of left $\CH$-comodule algebras
        \[\Theta\colon \ol N\lcp_{\sigma\inv}\CH^{\cop}\to (\CH\ou{\ol B}N\#_\sigma\CH)^{\co\CH};\quad \ol n\lcp X\mapsto X_+\ot \rho(X_-)n\]
        \item Assume $Q$ is a comodule algebra. Then \eqref{ehresmann precursor} is an isomorphism of algebras if we endow the source with the multiplication given by
        \[(\ol n\ot q)(\ol m\ot q')=\ol{\sigma\inv(q\o,q'\o)(q\t\la m)n}\ot q\z q'\z.\]
  
        \item If we identify $P\inv$ with $\ol N\#_{\sigma\inv}\CH^{\cop}$, the homomorphism \eqref{inverse extension morph} corresponds to
        \[\CH\ni X\mapsto 1\ot X\o\ot 1\ot X\t\in (\ol N\#_{\sigma\inv}\CH)\times_N(N\#_\sigma\CH).\]  
      \item The instance 
     \[(Q\ou{\ol B}P)^{\co\CH}\ou N P\to Q\ou{\ol B}P\]
     of the counit of the adjunction in \Cref{hopf module adjunction} is a bijection.
\end{enumerate}
\end{lem}
\begin{proof}
 For (1),   we know that $P\cong \CH\ou BN$ as right $\CH$-comodule and right $N$-module. Thus we have an isomorphism 
    \[(Q\ou{\ol B}P)^{\co\CH} \cong (Q\ou{\ol B}(\CH\ou BN))^{\co\CH}\cong(\CH\ou BN)\Box^\CH \ol Q\cong \ol N\ou{\ol B}Q\]
    where $\ol Q$ denotes $Q$ with the reverse comodule structure; the last isomorphism is given by
    \begin{align*}
        (\CH\ou BN)\Box^\CH \ol Q\cong& \ol N\ou{\ol B}Q\\
        X\ot n\ot q\mapsto& \overline{n}\ot q\,\overline{\varepsilon(X)}\\
        q\rmo\ot n\ot q\rz\mapsfrom& \overline{n}\ot q.
    \end{align*}
    The composed isomorphism has the form stated in \eqref{ehresmann precursor}. It is obviously left $\ol N$-linear, and \begin{align*}\theta(\ol n\ot q)\ol m&=(q\rz\ot\rho(q\rmo)n)\ol m\\&=q\rz\ot m\rho(q\rmo)n
    \\&=q\rz\ot \rho(q\rmo{}_{[+]})(q\rmo{}_{[-]}\la m)n
    \\&=q\z\rz\ot\rho(q\z\rmo)(q\o\la m)n
    \\&=\theta(\ol{(q\o\la m)n}\ot q\z).
    \end{align*}

   For (2),     we consider the special case where $Q=\CH$, and get
    \[\Theta\colon \overline{N}\ou \BB\CH\ni \overline{n}\ot X\mapsto X_+\ot\rho(X_-)n\in (\CH\ou{\ol  B} P)^{\co \CH}=P\inv,\]
    and this map is also left $\CH$-colinear. In particular $P\inv$ as a left $\CH$-comodule algebra with coinvariants $\ol N$ has a normal basis, hence is a crossed product. The action in it is the same as the original action of $\CH$ on $N$ by the linearity statement in (1). We denote the cocycle by $\alpha:\CH^{\cop}\ot_{B} \CH^{\cop}\to \ol N$. 
    
 The multiplication formula in (3) follows by again considering the identification of $(Q\ou{\ol B}P)^{\co\CH}$ with $\ol N\ou{\ol B}Q $ and using the formula for multiplication in the crossed product $\ol N\#_{\sigma\inv}\CH^{\cop}$.

    To see that \eqref{inverse extension morph} corresponds to the claimed map in (4) we check
    \begin{align*}
        X_+\ot X_-{}\yi{}\ot X_-{}\er{}
        &=X_+\ot \rho(X_{-[+]})\ot \trho(X_{-[-]})\\
        &=X\o{}_{+}\ot \rho(X\o{}_{-})\ot\trho(X\t)\\
        &=\Theta(1\#X\o )\ot (1\ot X\t).
    \end{align*}
    The morphism  
    \[f\colon \CH\cong P\inv\Box^{L(P,\CH)}P\to (\ol N\ou{\ol B}\CH)\Box^{L(P,\CH)}(N\ou B\CH)\]
    we just described is an algebra map. We have
    \begin{align*}
        f(X)f(Y)&=(1\ot X\o)(1\ot Y\o)\ot (1\ot X\t)(1\ot Y\t)\\
        &=\alpha(X\t,Y\t)\ot X\o Y\o\ot\sigma(X\th,Y\th)\ot X\fo Y\fo\\
        f(XY)&=1\ot X\o Y\o\ot 1\ot X\t Y\t.
    \end{align*}
    Applying $\varepsilon$ to the $\CH$ factors yields \[\alpha(X\o,Y\o)\sigma(X\t,Y\t)=\varepsilon(XY)\]
    which is  enough to show that $\alpha=\sigma\inv$ by \Cref{monoidal functor}, as claimed. In particular the cocycle in the crossed product $P\inv$ is also invertible, and thus $P\inv$ is cleft.

    To prove (5) we precompose the counit of the adjunction with the map \eqref{ehresmann precursor}, we get
    \[Q\di P\cong (\ol N\ou{\ol B}Q)\ou N P\xrightarrow{\theta\ot 1}(Q\ou {\ol B}P)^{\co\CH}\ou NP\rightarrow Q\ou{\ol B}P.\]
    With $p=n\trho(X)$, the map is given by 
    \[q\ot p\mapsto \ol n\ot q\ot \trho(X)\mapsto q\rz\ot\rho(q\rmo)n\ot \trho(X)\mapsto q\rz\ot\rho(q\rmo)n\trho(X)=q\rz\ot\rho(q\rmo)p. \]
    This map is indeed a bijection with the inverse
    \[Q\ou{\ol B}P\ni q\ot p\mapsto q\z\ot \tilde\rho(q\o)p\in Q\di P.\]
    Indeed
    \begin{align*}
        q\rz\z\ot \trho(q\rz\o)\rho(q\rmo)p&=q\z\ot\trho(q\o{}_{+})\rho(q\o{}_{-})p=q\ot p\\
        q\z\rz\ot\rho(q\z\rmo)\trho(q\o)p
        &=q\rz\ot \rho(q\rmo{}_{[+]})\trho(q\rmo{}_{[-]})p=q\ot p.
    \end{align*}
\end{proof}

\begin{rem}
    We have shown that the passage from a crossed product $P$ to $P\inv$ amounts to keeping the action (interpreting it as an action of $\CH^\cop$ on $N^\op$) and inverting the cocycle (interpreting it as a cocycle on $\CH^\cop$ with values in $N^\op$). Of course we can apply this procedure in the coopposite situation, then passing from a left cleft extension to an anti-cleft right extension. We then get $(P\inv)\inv=P$.
\end{rem}
\begin{lem}Assume the hypotheses of \Cref{calculating inverse part one}.
\begin{enumerate}
       \item We have an isomorphism 
     \[\Phi\colon Q\ou{\ol B}\ol N\to \left(Q\ou{\ol B}(N\#_\sigma\CH)\right)^{\co\CH};\quad q\ot\ol n\mapsto q\rz\ot n\ot q\rmo\]
     which is an isomorphism of algebras if we endow the source with the multiplication 
       \[(p\ot \overline{n})(q\ot \overline{m})=  p\rz q\rz\ot \overline{m(q\rmt\la n)\sigma(q\rmo,p\rmo)}.\]
     In particular $P\inv\cong \CH\ou \BB \ol N$ with multiplication 
     \begin{align}\label{equ. weird left crossed product}
         (X\ot \overline{n})(Y\ot \overline{m})= X_+Y_+\ot \overline{m(Y_{-}\o\la n)\sigma(Y_{-}\t,X_-)}
     \end{align}

     \item $P\cong \CH\ot_{B}N$ with multiplication 
     \begin{equation}\label{weird crossed product}(X\ot n)(Y\ot m)=X_{[+]}Y_{[+]}\ot \sigma^{-1}(Y_{[-]}\o, X_{[-]})\,(Y_{[-]}\t\la n)\,m.\end{equation}which we call weird crossed product.
\end{enumerate}

\end{lem}
\begin{proof}
    We consider the isomorphism 
    \[(Q\ou{\ol B}(N\#_\sigma\CH))^{\co\CH}\cong (N\#_\sigma\CH)\Box^\CH \ol Q\cong N \ou B\ol Q\] which has the claimed form after  identifying $Q\ot_{\BB}\ol N$ with $N\ot_{B} \ol Q$.         To check the algebra map property we calculate, under the same identification,
        \begin{align*}
            \Phi(n\ot \ol p)\Phi(m\ot \ol q)
            &=(p\rz\ot n\ot p\rmo)(q\rz\ot m\ot q\rmo)\\
            &=p\rz q\rz\ot (m\ot q\rmo)(n\ot p\rmo)\\
            &=p\rz q\rz\ot m(q\rmth \la n)\sigma(q\rmt,p\rmt)\ot q\rmo p\rmo\\
            &=\Phi( m(q\rmt\la n)\sigma(q\rmo,p\rmo)\ot \ol {p\rz q\rz})
        \end{align*}

    For the second statement consider that $P=(P\inv)\inv$ so that we can apply the previous result in the co-opposite version.
\end{proof}

\begin{rem}
  To recover the action and cocycle of a crossed product, one can simply apply the projection $\pi\colon N\#_\sigma\CH\to N, n\ot X\mapsto n\varepsilon(X)$ to suitable products, namely
  \[X\la b=\pi((1\ot X)(b\ot 1)),\quad \sigma(X,Y)=\pi((1\ot X)(1\ot Y)).\]
  This can be viewed as a special case of \Cref{free comodule bijection}:
  \[\Hom^\CH(\CH,N\ou B\CH)\cong\Hom_{-B}(\CH,N)\]
  maps $X\mapsto (1\ot X)(n\ot 1)$ for fixed $n$ to $X\mapsto X\la b$, and
  \[\Hom^\CH(\CH\ou \BB\CH,N\ou B\CH)\cong \Hom_{-B}(\CH\ou\BB\CH,N)\]
  maps $(X\ot Y)\mapsto \pi((1\ot X)(1\ot Y))$ to $\sigma$.
  
   In the same way the  action and cocycle in weird crossed product formulas like \eqref{equ. weird left crossed product} and \eqref{weird crossed product} can be recovered from the multiplication  using the maps in \Cref{skew bijection for colinear maps}. Namely, identifying $P=\CH\ou BN$ along the weird normal basis
   \begin{align*}
       \Hom_{B-}(\CH,N)&\cong\Hom^\CH(\CH,\CH\ou BN)\\
       (Y\mapsto Y\la n)&\mapsto (Y\mapsto (1\ot n)(Y\ot 1))\\
       \Hom_{B-}(\CH\ou{\BB}\CH,N)&\cong\Hom^\CH(\CH\ou B\CH,\CH\ou BN)\\
       \sigma\inv&\mapsto(X\ot Y\mapsto (X\ot 1)(Y\ot 1)).
   \end{align*}
 Note though that \Cref{skew bijection for colinear maps} makes a faithful flatness assumption on $N$ for establishing the relevant bijections, which we do not need here.
\end{rem}
\begin{rem}
    Consider the special case where $\CH$ is an ordinary Hopf algebra. Then the product formula on $\CH\ou{\ol B}\ol N=(N\#_\sigma\CH)\inv$ reads
    \begin{align*}
        (X\ot\ol n)(Y\ot \ol m)
        &=X\o Y\o\ot \ol{\sigma(S(Y\t)\t,S(X\t))}(\ol{S(Y\t)\o\la n})\ol m\\
        &=X\o Y\o\ot\ol{\sigma(S(Y\t),S(X\t))}(\ol{S(Y\th)\la n})\ol m\\
        &=X\o Y\o\ot\kappa(X\t,Y\t)(\ol n\triangleleft Y\th)\ol m
    \end{align*}
    if we put $\ol n\triangleleft X=\ol{S(X)\la n}$  and $\kappa(X,Y)=\ol{\sigma(S(Y),S(X))}$. This is the formula for an ordinary left crossed extension (a cleft left comodule algebra) with $\CH$ acting on the right on the coinvariant subalgebra, here $\ol N$. We pointed out already in \Cref{first left cleft remark} that this usual Hopf algebra left crossed product does not match our definition of a left cleft extension. Nonetheless we see thus that the formula \eqref{equ. weird left crossed product} does generalize the usual left crossed product.
\end{rem}

\begin{LD}
    Let  $(\sigma,\la)$ be  an invertible  2-cocycle of $\CH$ over $N$. The associated \emph{Connes-Moscovici Hopf algebroid} $N\#_{\sigma}\CH\#_{\sigma}N$ is the vector space $N\ot_{B} \CH\diamond_{B} N$ with the $N$-coring structure:
    \[s(n)=n\ot 1\ot 1,\quad t(n)=1\ot 1\ot n,\quad \]
    \[\Delta(n\ot X\ot m)=n\ot X\o\ot 1\ot 1\ot X\t\ot m,\quad \varepsilon(n\ot X\ot m)=n\varepsilon(X)m,\]
and the algebra structure:
    \[(n\ot X\ot m)(n'\ot Y\ot m')=n(X\o\la n')\sigma(X_{\scriptscriptstyle{(2)}}, Y\o)\ot X_{\scriptscriptstyle{(3)}}Y_{\scriptscriptstyle{(2)}}\ot \sigma^{-1}(X_{\scriptscriptstyle{(4)}}, Y_{\scriptscriptstyle{(3)}})(X_{\scriptscriptstyle{(5)}}\la m')m,\]
    for all $X,Y\in \CH$ and $n, n', m, m'\in N$. The translation maps are given by
    \[(n\ot X\ot m)_{+}\ot(n\ot X\ot m)_{-}=(n\ot X_{+}\o\ot X_{+}\t\wedge 1)\ot (X_{-}\o\vee m\ot X_{-}\t\ot 1)\]
    \[(n\ot X\ot m)_{[+]}\ot(n\ot X\ot m)_{[-]}=(X_{[+]}\o\vee 1\ot X_{[+]}\t\ot  m)\ot (1\ot X_{[-]}\o\ot X_{[-]}\t\wedge n).\]
\end{LD}

That the Connes-Moscovici Hopf algebroid is in fact a Hopf algebroid (in particular a bialgebroid) is a consequence of the following Theorem. Note that if we assumed that $P$ is suitably faithfully flat, then  the Theorem would just compute the Ehresmann Hopf algebroid recalled in \Cref{lem. bigalois given by ehresmann groupoid} and establish that it is isomorphic to the Connes-Moscovici Hopf algebroid. For a cleft extension, the faithful flatness required would already follow from faithful flatness assumptions on $\CH$.  Without the faithful flatness assumption we cannot refer to that Ehresmann Hopf algebroid, but in our cleft situation it can nonetheless be explicitly constructed. 
\begin{thm}\label{Ehresmann in the cleft case}
    Let $(\la,\sigma)$ be an invertible cocycle of $\CH$ over $N$. The Connes-Moscovici  Hopf algebroid $\CL$  is an $N$-Hopf algebroid. It is the   left Ehresmann Hopf algebroid $L(N\#_\sigma\CH,\CH)$ of the right anti-Galois extension $N\#_\sigma\CH$ in the following sense: A left coaction of  $\CL$ on $N\#_\sigma\CH$ is given by
    \begin{align*}
      N\#_\sigma\CH&\to N\#_\sigma\CH\#_\sigma N\times_N N \#_\sigma\CH\\
      n\# X&\mapsto n\ot X\o\ot 1\ot 1\ot X\t.
    \end{align*}
    Any left coaction on $N\#_\sigma\CH$ of an $N$-bialgebroid $\mathcal G$ that commutes with the right $\CH$-coaction corresponds to a bialgebroid map $ \CL \to\mathcal G$. 
   
    The Galois map
    \[(N\#_\sigma\CH)\ou {\ol B}(N\#_\sigma\CH)\to \CL\diamond_N(N\#_\sigma\CH)\]
    is a bijection. Thus if $^{\co \CL}(N\#_\sigma\CH)=\ol B$, then $N\#_\sigma\CH$ is an $\CL$-$\CH$-biGalois extension.

    In the same sense $\CL$ is the right Ehresmann Hopf algebroid of the left $\CH$-Galois extension $\ol N\#_{\sigma\inv}\CH^{\cop}$ under the right coaction given by
    \begin{align*}
        \ol N\#_{\sigma\inv}\CH^{\cop}&\to N\#_\sigma\CH\#_\sigma N\\
        \ol n\# Y&\mapsto 1\# X\o\ot 1\ot X\t\ot n.
    \end{align*}
 \end{thm}
\begin{proof}
   Put $P=N\#_\sigma\CH$. We define $\CL:=L(P,\CH)=(P\ou{\ol B}P)^{\co\CH}$ with the algebra structure coming from multiplication on the left tensor factor and the opposite multiplication on the right factor, as in the general Ehresmann Hopf algebroid construction \Cref{lem. bigalois given by ehresmann groupoid}. We have shown in \Cref{calculating inverse part one} that the counit of the adjunction in \Cref{hopf module adjunction} defines an isomorphism  $\kappa\colon\CL\diamond_NP\to P\ou{\ol B}P$ given by $(p\ot q)\ot q'\mapsto p\ot qq'$. Therefore we can proceed exactly as in \cite{HS25} to define a coaction of $\CL$ on $P$ 
   by
    \[\delta=(P\xrightarrow{p\mapsto p\ot 1}P\ou{\ol B}P\xrightarrow{\kappa^{-1}}\CL\diamond_NP).\]
   The image of $\delta$ is contained in $\CL\times_NP$ because if $\delta(p)=p'\ot p''\ot p'''$ (an element characterized by the property $p'\ot p''p'''=p\ot 1$), then $\kappa(p'\ot np''\ot p''')=p'\ot np''p'''=p\ot n=p'\ot p''p'''n=\kappa(p'\ot p''\ot p'''n)$. The corestriction of $\delta$ to said Takeuchi product is an algebra map because if also $\delta(q)=q'\ot q''\ot q'''$ then $\kappa((p'\ot p'')(q'\ot q'')\ot p'''q''')=\kappa(p'q'\ot q''p''\ot p'''q''')=p'q'\ot q''p''p'''q'''=pq'\ot q''q'''=pq\ot 1$. The same universal property of $\CL$ as used in \cite{HS25} follows from the adjunction in \Cref{hopf module adjunction}:
   \[\Hom^\CH(P,\mathcal G\diamond_N P)\cong\Hom^\CH_{-P}(P\ou{\ol B}P,\mathcal G\diamond_N P)\cong\Hom_N((P\ou{\ol B}P)^{\co\CH},\mathcal G).\]
   The universal property can then be used to define the comultiplication $\CL\to\CL\diamond_N\CL$ as in \cite{HS25}, and the Galois map $P\ou{\ol B}P\to\CL\diamond_NP$ is $\kappa^{-1}$.

   In \cite[Lemma 3.21]{HS25}, it was that bijectivity of this Galois map implies that $\CL$ is left Hopf if $P$ is left faithfully flat over $N$. But of this hypothesis the proof of \cite[Lemma 3.21]{HS25} only uses that $(\text{---})\ou NP$ \emph{reflects} isomorphisms; in our case, by \Cref{Doi-integrals}, $N\subset P$ is a direct summand as a left $N$-module so this is true.

   It remains to show that $\CL$ identifies with the Connes-Moscovici Hopf algebroid as claimed. From (1) and (3) of \Cref{calculating inverse part one}  we know that 
   \begin{align*}
       \CL\cong \ol N\ou{\ol B}P\cong \ol N\ou{\ol B}(N\#_\sigma\CH)\cong N\ot_{B} \CH\diamond_{B} N,
   \end{align*}
   and that the corresponding multiplication is as claimed. The isomorphism \eqref{ehresmann precursor} translates to 
    \[\theta:N\#_{\sigma}\CH\#_{\sigma}N\to L(N\#_{\sigma}\CH,\CH),\quad n\ot X\ot m\mapsto n\trho(X_{+})\ot\rho(X_{-})m,\]
    and the same formula describes the translation map for the left $\CL$-Galois extension $P$.
    Since \begin{multline*}\kappa(\theta(n\ot X\o\ot 1)\ot 1\ot X\t)=n\trho(X\o{}_{+})\ot\rho(X\o{}_{-})\trho(X\t)
    \\=n\trho(X_+)\ot\rho(X_{-[+]})\trho(X_{-[-]})=n\trho(X)\ot 1,\end{multline*}
    the coaction is as claimed.
   We note that the inverse of $\theta$ is given by
    \[\phi:L(N\#_{\sigma}\CH,\CH)\to N\#_{\sigma}\CH\#_{\sigma}N,\quad n\ot X_{+}\ot m\ot X_{-}\mapsto n\ot X\o\ot X\t\wedge m.\]

   To find the explicit form of the left translation map we recall from the proof of \cite[Lem.~3.21]{HS25} that 
        \[
\begin{tikzcd}
  &P\ot_{\BB} P\ot_{\BB}P \arrow[d, "P\ot \lcan"] \arrow[r, "\lcan\ot P"] & \cL\diamond_N P \ot_{\BB}P \arrow[dd, "\cL\ot \lcan"] &\\
  &P\ot_{\BB} (\cL \diamond_N P) \arrow[d, " \lcan_{1,3}"]\quad&\qquad&\\
   & (\cL\ot_{\ol N} \cL )\diamond_NP   \arrow[r, "\lambda\ot P"] & (\cL\diamond_N \cL) \diamond_NP &
\end{tikzcd}
\]
commutes. Thus we can compute $\lambda\inv$ by chasing $n\ot X\ot m\ot 1_\CL\ot 1_P$ from the bottom right corner of the diagram to the bottom left corner. 

To prepare, we note first that in $N\#_\sigma \CH$
\begin{multline*}
    X\o\vee m\ot X\t=\rho(X\o{}_{+})m\rho(X\o{}_{-})\trho(X\t)=\rho(X_+)m\rho(X_{-[+]})\trho(X_{-[-]})=\rho(X)m
\end{multline*}
and therefore
\[\delta(\rho(X)m)=\delta(X\o\vee m\ot X\t)=X\o\vee m\ot X\t\ot 1\ot 1\ot X\th.\]
Now
\begin{align*}
     \lcan_{1,3}&(P\ot\lcan)(\lcan\inv\ot P)(\CL\ot\lcan\inv)(n\ot X\ot m\ot 1_\CL\ot 1_P)
    \\
    =&\lcan_{1,3}(P\ot\lcan)(\lcan\inv\ot P)(n\ot X\ot m\ot 1_P\ot 1_P)
    \\
    =&\lcan_{1,3}(P\ot\lcan)(n\trho(X_+)\ot\rho(X_-)m\ot 1_P)
    \\
    =&\lcan_{1,3}((n\trho(X_+))\ot ( X_{-}\o\vee m\ot X_{-}\t\ot 1)\ot (1\ot X_{-}\th))
    \\
    =&(n\ot X_+\o\ot 1)\ot (X_-\o\vee m\ot X_-\t\ot 1)\ot (\trho(X_+\t)\trho(X_-\th))
    \\
    =&(n\ot X\o\ot 1)\ot (X\t{}_{-}\o\vee m\ot X\t{}_{-}\t\ot 1)\ot \trho(X\t{}_{+})\trho(X\t{}_{-}\th)
     \\
     =&(n\ot X\o\ot 1)\ot (X\t{}_{-}\o\vee m\ot X\t{}_{-}\t\ot 1)\ot \trho(X\t{}_{++})\trho(X\t{}_{+-})
      \\
      =&(n\ot X\o\ot 1)\ot (X\t{}_{-}\o\vee m\ot X\t{}_{-}\t\ot 1)\ot X\t{}_{+}\wedge 1
       \\
       =&(n\ot X\o\ot X\t{}_{+}\wedge 1)\ot (X\t{}_{-}\o\vee m\ot X\t{}_{-}\t\ot 1)\ot 1_P
       \\
       =&(n\ot X{}_{+}\o\ot X{}_{+}\t\wedge 1)\ot (X{}_{-}\o\vee m\ot X{}_{-}\t\ot 1)\ot 1_P.
\end{align*}

    By symmetric reasoning, $P\inv=\ol N\#_{\sigma\inv}\CH^{\cop}$ is a right comodule algebra over the same Connes-Moscovici algebroid, for which the right anti-Galois canonical map is bijective; therefore $\cL$ is also anti-left Hopf algebroid. We omit the analogous proof that the translation map in this case is also as stated.
    \end{proof}
    \begin{rem}
    Under suitable flatness conditions, 
    the monoidal functor from \Cref{monoidal functor} factors over an equivalence $\LComod\CH\to\LComod{L(P,\CH)}$, see \Cref{bigalois imply equivalence}. In our present situation it is not hard to check that the functor factors over $\LComod{L(P,\CH)}$ though it may not give an equivalence.
    \end{rem}

\subsection{Drinfeld twist of Hopf algebroids}\label{sec. Drinfeld twist of Hopf algebroids}

In this section, we will study the Drinfeld twist theory of Hopf Galois extensions. In \Cref{defi. 2 cocycle and twisted module algebra}, consider the special case that $N=B$, then we can construct a Drinfeld cotwist bialgebroid, which generalizes the 2-cocycle cotwist in \cite{Boehm, HM22, HM23} (where 2-cocycles factor through $\CH\ot_{B^e}\CH$).

\begin{LD}\label{def. 2-cocycle on L}
    Let $\CH$ be a left $B$-bialgebroid. The following data are equivalent:
    \begin{enumerate}
        \item A cocycle $(\Gamma,\la)$ over $N=B$ in the sense of \Cref{defi. 2 cocycle and twisted module algebra}
        \item A map $\Gamma\in {}_{B^{e}-}\Hom(\CH\otimes_{\BB}\CH, B)$ such that 
        \begin{align}\Gamma(X, \Gamma(\one{Y}, \one{Z})\two{Y}\two{Z})&=\Gamma(\Gamma(\one{X}, \one{Y})\two{X}\two{Y}, Z)\label{i}\\
        \Gamma(1_{\CH}, X)&=\varepsilon(X)=\Gamma(X, 1_{\CH})\label{ii}
    \end{align} 
for all $X, Y, Z\in \CH$ and $b\in B$.
    \end{enumerate}
 $\Gamma$ is invertible in the sense of \Cref{defi. 2 cocycle and twisted module algebra} iff  there is a map $\Gamma^{-1}\in {}_{B^{e}-}\Hom(\CH\otimes_{B}\CH, B)$, such that
 \begin{align}
         \Gamma^{-1}(X, \overline{b}Y)&=\Gamma^{-1}(X\o, Y)\Gamma^{-1}(X\t, \overline{b}),\label{iii}\\
        \Gamma(X\o,Y\o)\Gamma^{-1}(X\t, Y\t)&=\Gamma(X, \varepsilon(Y)), \label{iv}\\ 
        \Gamma^{-1}(X\o,Y\o)\Gamma(X\t, Y\t)&=\varepsilon(XY).\label{v} 
    \end{align} 
    We will call such a map $\Gamma$ an (invertible) 2-cocycle on $\CH$.
\end{LD}

\begin{proof}
   Indeed, given a 2-cocycle $\Gamma$ on $\CH$, if we define $X\la b=\Gamma(X, b)$, it is not hard to check \eqref{2-cocycle 1} and \eqref{2-cocycle 2}. For \eqref{2-cocycle 3}, we have
    \begin{align*}
        X\la(mn)=&\Gamma(X, mn)=\Gamma(X,\Gamma(m,n))=\Gamma(\Gamma(X\o, m)X\t, n)=\Gamma(X\o, m)\Gamma(X\t, n)\\
        =&(X\o\la m)\, (X\t\la n).
    \end{align*}
    Also, it is not hard to check (\ref{2-cocycle 4}) and (\ref{2-cocycle 5}). For (\ref{2-cocycle 6}), we have
    \begin{align*}
        (X\o\la b)\Gamma(X\t, Y)=&\Gamma(X\o, b)\Gamma(X\t, Y)=\Gamma(X, \Gamma(b, Y\o)Y\t)=\Gamma(X, bY).
    \end{align*}
    With the help of (\ref{2-cocycle 6}), it is not hard to check (\ref{2-cocycle 7}) and (\ref{2-cocycle 8}). In addition, if $\Gamma$ is an invertible 2-cocycle on $\CH$, we can see (\ref{2-cocycle 9}), (\ref{2-cocycle 11}) and (\ref{2-cocycle 12}) can also be recovered by the axioms on $\Gamma^{-1}$.
    Moreover, we have $\Gamma^{-1}(X, \overline{b})=\Gamma(X, b)$. Indeed,
    \begin{align*}
        \Gamma^{-1}(X, \overline{b})=\Gamma(X\o, 1)\Gamma^{-1}(X\t, \overline{b})=\Gamma(X, b),
    \end{align*}
    where the 2nd step uses \eqref{iv}. As a result, (\ref{2-cocycle 10})  can be recovered by \eqref{iii}. By (\ref{equ. right cocycle}), we can see $\Gamma^{-1}$ safisfies 
    \[\Gamma^{-1}(X, \overline{\Gamma^{-1}(Y\t, Z\t)}Y\o\, Z\o)=\Gamma^{-1}(\overline{\Gamma^{-1}(X\t, Y\t)}X\o\, Y\o, Z).\]
    
    Conversely, given a (invertible) 2-cocycle of $\CH$ over $B$ as in Definition~\ref{defi. 2 cocycle and twisted module algebra}, we can see the action satisfies $X\la b=\sigma(X, b)=\sigma^{-1}(X,\overline{b})$ by (\ref{2-cocycle 6}) and (\ref{2-cocycle 10}) of Definition~\ref{defi. 2 cocycle and twisted module algebra}. Therefore, all the axioms of Definition~\ref{def. 2-cocycle on L} can be checked. So we can see the two definition are equivalent in the case that $N=B$.   \end{proof}
  
\begin{rem}  
    In \cite{Boehm,HM22,HM23} the additional condition is imposed that the cocycle $\Gamma$ factors over $\CH\ou{B^e}\CH$.  As it turns out as a result of our general theory of cleft extensions and Connes-Moscovici-Ehresmann Hopf algebroids, the extra condition is simply not needed to arrive at essentially the same conclusions.  
    
    By Remark \ref{rem. bilinear 2-cocycle of L over N} the addtional condition is equivalent to "triviality" of the action in the sense that $X\la b=\varepsilon(Xb)$, so that Definition \ref{def. 2-cocycle on L} with this triviality added is indeed the one given in \cite{Boehm, HM22, HM23}.
\end{rem}

Our definition of a cocycle on $\CH$ is a special case of the definition of a cocycle in \Cref{defi. 2 cocycle and twisted module algebra}. Thus, we also get a cleft extension $_\Gamma\CH:=B\#_\Gamma \CH$, which identifies with $\CH$ as a right comodule algebra, with the multiplication given by 
\[X\cdot Y=\Gamma(X\o,Y\o)X\t Y\t.\]
The left Ehresmann $\CH^\Gamma:=L(\CH_\Gamma,\CH)$ can be identified with $\CH$ with the new multiplication
\begin{align}\label{twistprod}
    X\cdot_{\Gamma} Y:=\Gamma(\one{X}, \one{Y})\overline{\Gamma^{-1}(\three{X}, \three{Y})}\two{X}\two{Y}.
\end{align}
In particular, ${}_{\Gamma}\CH$ is a $\CH^{\Gamma}$-$\CH$-biGalois extension.

\begin{cor}\label{thm. one side twist is cleft extension}
  Let $\CH$ be a $B$-Hopf algebroid. Then $B\subseteq P$ is an anti-cleft extension over $\CH$ if and only if $P\cong {}_{\Gamma}\CH$ for some invertible 2-cocycle $\Gamma$ on $\CH$. 
\end{cor}

\begin{lem}
    Let $\cL$ be a left bialgebroid over $B$, and $\Gamma$ be an invertible  2-cocycle on $\cL$. Then  ${}^{\cL}\mathcal{M}\cong{}^{\cL^{\Gamma}}\mathcal{M}$ as monoidal categories.
\end{lem}

 More precisely,   given $V$ and $W$ as left $\cL$-comodule, then $V\ot_{B}{}W$ is also left $\cL$-comodule with diagonal coaction and $B$-bimodule structure given by $b(v\ot w)b'=bv\ot wb'$ for any $b,b'\in B$ and $v\ot w\in V\ot_{B}{}W$. The monoidal functor is the identity on the underlying $B$-bimodule, with the corresponding natural transformation given by
    \[\xi_{V, W}:v\ot_{B}{}w\mapsto \Gamma(v\mo, w\mo)v\z\ot_{B}{}w\z,\]
    for any $v\ot w\in V\ot_{B}{}W$. The coherence condition corresponds to the 2-cocycle condition on $\Gamma$. Note that this monoidal functor structure was already discussed in a more general setting in \Cref{monoidal functor}. Here the functor between comodule categories is in fact an equivalence without any further assumptions, because on comodules it is simply the identity, given that the coring structure of $\CH$ does not change.

\begin{cor}\label{lem. two side cleft extensions}
   Let $\CH$ be a $B$-Hopf algebroid and $\Gamma$ be an invertible 2-cocycle on $\CH$. Then ${}_{\Gamma}\CH$ is a $\CH^\Gamma$-$\CH$-biGalois extension which is both left and right cleft. $\Gamma\inv$ is a cocycle on $\CH^\Gamma$ with $(\CH^\Gamma)^{\Gamma\inv}=\CH$.
\end{cor}
\begin{proof}
    We have proved that ${{}_{\Gamma}\CH}$ is biGalois in \Cref{Ehresmann in the cleft case} (there is no issue with the coinvariant subalgebra here). Obviously it has normal basis as a left comodule algebra as well, so it is left cleft. It is not hard to check that as a crossed product left comodule algebra over $\CH^\Gamma$, the algebra $_\Gamma\CH$ is defined by $\Gamma$ considered as a cocyle on $(\CH^\Gamma)^\cop$.  In other words, if we consider $_\Gamma\CH\cong B\#_{\Gamma}\,\CH$ as a left cleft $\CH^{\Gamma}$ extension, $_\Gamma\CH\cong \BB\#_{\Gamma}\,(\CH^{\Gamma})^{\cop}$.
\end{proof}

\subsection{Twisting CM-bialgebroids}
Under suitable flatness conditions on all Hopf algebroids and extensions involved, it was shown in \cite{HS25} that bi-Galois extensions form a groupoid. In particular, for two $\CH$-anti-Galois extensions $N\subset P$ and $A\subset Q$ we can simply form $Q\Box^\CH P\inv$  which is a $L(Q,\CH)$-$L(P,\CH)$-bi-Galois extension. Most of this works also without any flatness assumptions in our situation, using the explicit nature of cleft extensions instead.

\newcommand\laa{\underline\triangleright}
\begin{thm}\label{cm-bigalois}
    Let $(\la,\sigma)$ and $(\laa,\kappa)$ be two cocyles of $\CH$ over $N$ and $A$, respectively. Then
    $A\#_{\kappa}\CH\#_{\sigma}N:=A\ot_{B} \CH\di N$
    with multiplication
     \begin{align}\label{equ. twist on two side products1}
         (a\ot X\ot m)&(b\ot Y\ot n)\\&\notag=a(X\o\laa b)\kappa(X_{\scriptscriptstyle{(2)}}, Y\o)\ot X_{\scriptscriptstyle{(3)}}Y_{\scriptscriptstyle{(2)}}\ot \sigma^{-1}(X_{\scriptscriptstyle{(4)}}, Y_{\scriptscriptstyle{(3)}})(X_{\scriptscriptstyle{(5)}}\la n)m,
      \end{align}
      is a $A\#_\kappa\CH\#_\kappa A$---$N\#_\sigma\CH\#_\sigma N$-bicomodule algebra with left and right coactions each given by
      \begin{align*}
          a\ot X\ot n&\mapsto a\ot X\o\ot 1\ot 1\ot X\t\ot n
      \end{align*}
      (with different targets of the two maps).
      $A$ maps to the right coinvariants of this algebra, and the right anti-Galois map
      \[(A\#_{\kappa}\CH\#_{\sigma}N)\ou A(A\#_{\kappa}\CH\#_{\sigma}N)\to (A\#_{\kappa}\CH\#_{\sigma}N)\diamond_N (N\#_\sigma\CH\#_\sigma N)\]
      is a bijection. Similarly $N$ maps to the left coinvariants and the left Galois map is a bijection.
\end{thm}
\begin{proof}
    The algebra $A\#_{\kappa}\CH\#_{\sigma}N$ defined above is the cotensor product $Q\Box^\CH P\inv$ with $P=N\#_\sigma\CH$ and $Q=A\#_\kappa\CH$. It is a $L(Q,\CH)$-$L(P,\CH)$-bicomodule algebra as indicated because the cotensor product is a split equalizer both because $Q$ and because $P\inv$ has normal basis, so there is no issue with defining the comodule structures. As for the right anti-Galois map, we can write it as a composition
    \begin{multline*}(Q\Box^\CH P\inv)\ou A(Q\Box^\CH P\inv)\xrightarrow\xi Q\Box^\CH(P\inv\ou BP\inv)\\\xrightarrow{Q\Box\rcan}Q\Box^\CH(P\inv\diamond_N L(P,\CH))\cong
    A\ou BP\inv\diamond_NL(P,\CH)
    \cong (Q\Box^\CH P\inv)\diamond_NL(P,\CH),\end{multline*}
    where $\xi$ is the isomorphism from \Cref{monoidal functor}. That the left Galois map is a bijection is analogous.
\end{proof}

\begin{lem}\label{one side twisted CM}
    In the situation of \Cref{cm-bigalois} assume that $N$ is a subalgebra of $A$, and that $A$ is the subalgebra of right coinvariants of $A\#_\kappa\CH\#_\sigma N$.  (For example this is the case if $\CH$ is right faithfully flat  over both copies of $B$, and $N$ is faithfully flat on both sides over $B$, because then we are in the situation of the biGalois theory developed in \cite{HS25}. It is also true if $A=N$. )
    
     Then $A\#_\kappa \CH\#_\sigma N$ is anti-cleft as right $N\#_\sigma\CH\#_\sigma N$-comodule algebra. In particular
    \[A\#_\kappa\CH\#_\sigma N\cong A\#_\Gamma(N\#_\sigma\CH\#_\sigma N)\]
    is a crossed product for the cocycle $(\blacktriangleright,\Gamma)$ of $N\#_\sigma\CH\#_\sigma N$ over $A$ given by
    \begin{align*}
        (m\ot X\ot n)\blacktriangleright a&=m(X\laa a)n\\
        \Gamma(m\ot X\ot n,m'\ot Y\ot n')&=m(X\o\laa m')\kappa(X\t,Y\o)\sigma\inv(X\th,Y\t)(X\fo\la n')n.
    \end{align*}    
\end{lem}
\begin{proof}
    It is rather obvious that $A\#_\kappa\CH\#_\sigma N$ has normal basis as a right anti-Galois extension, and thus it is a crossed product as indicated. The action and cocycle can be computed by using the projection 
    \[\Pi\colon A\#_\kappa\CH\#_\sigma N\to A;\quad a\ot X\ot n\mapsto a\varepsilon(X)n\]
    namely $\xi\blacktriangleright a=\Pi(\xi(a\# 1\# 1))$ and $\Gamma(\xi,\xi')=\Pi(\xi\xi')$ implicitly (and non-multiplicatively) injecting $N\#_\sigma\CH\#_{\sigma}N$ into $A\#_\kappa\CH\#_\sigma N$.
    \end{proof}
\begin{cor}
    In the situation of \Cref{one side twisted CM} assume $A=N$. Then $N\#_\kappa\CH\#_\sigma\CH$ is cleft as both a left and right comodule algebra. Therefore $N\#_\kappa\CH\#_\kappa N$ is a Drinfeld twist of $N\#_\sigma\CH\#_\sigma\CH$ by \Cref{lem. two side cleft extensions}.
\end{cor}

\begin{rem} 
If $N\subset A$ is a subalgebra, and the action of $\CH$ on $N$ is not the restriction of the action on $A$, then the cocycle $\Gamma$ is nontrivial even if both $\sigma$ and $\kappa$ are trivial. Also, in this case the action of  $N\#_\sigma\CH\#_\sigma N$ on $A$ does not satisfy the equivalent conditions of \Cref{rem. bilinear 2-cocycle of L over N}. 

    If $A=N$, but the actions of $\CH$ are different, this means that the restriction of the action of $N\#_\sigma\CH\#_\sigma N$ on $A$ to $N$ is not the trivial action, which is $(n\ot X\ot m)\la n'=\varepsilon((n\ot X\ot m)n')=n\varepsilon(X\la n')m$. Thus our generalization of a Drinfeld twist is in fact strictly more general than the version in \cite{Boehm,HM22,HM23}.

    Examples of such situations can already rather easily be constructed in the case where $\CH$ is an ordinary Hopf algebra (while $N\#_\sigma\CH\#_\sigma H$ would then still be a Hopf algebroid). 

    As a somewhat counterintuitive example consider a $k$-Hopf algebra $\CH$ and a Hopf crossed product $P=N\#_\sigma \CH$ with nontrivial action and cocycle. Consider any algebra $A$ with $N\subset A$ and the trivial action of $\CH$ on $A$ with trivial cocycle. Then $R=A\#_\kappa\CH\#_\sigma N=A\ot (N\#_\sigma\CH)\inv$ is just a tensor product of algebras, but it is a crossed product 
    \[R=A\#_\Gamma L(N\#_\sigma\CH)\]
    in which the action does not restrict to the trivial action on $N$, and  and the cocycle is nontrivial.
\end{rem}

\appendix

\section{Cleft extension for  Hopf algebroids with bijective antipodes} In this section we are going to show when $\CH$ is a Hopf algebroid with a bijective antipode in the sense of \cite{BS}, the two definition of cleft extensions concide. Recall that

\begin{defi}\label{def:right.bgd} Let $A$ be a unital algebra. A right bialgebroid over $A$ (or right $A$-bialgebroid) is an algebra $\cR$ with  algebra maps $s_{R}:A\to \cR$ and $t_{R}:A^{op}\to \cR$ whose images commute,  and a $A$-coring for the bimodule structure given by $a.X.a'=X\,s_{R}(a')\,t_{R}(a)$ which is compatible in the sense
\begin{itemize}
\item[(i)] The coproduct $\Delta_{R}$ corestricts to an algebra map  $\cR\to \cR{}_{A}\times \cR$ where
\begin{equation*} \cR{}_{A}\times \cR :=\{\ \sum_i X_i \blacklozenge_{A} Y_i\ |\ \sum_i s_{R}(a)X_i \ot Y_i=
\sum_i X_i \ot  t_{R}(a)Y_i  ,\quad \forall a\in A\ \}\subseteq \cR\blacklozenge_{A}\cR,
\end{equation*}
 is an algebra via factorwise multiplication. We use the upper sumless Sweedler index to denote the coproduct, namely, $\Delta_{R}(X)=X\teins{}\ot X\tzwei{}$.
\item[(ii)] The counit $\varepsilon$ is a right character in the following sense:
\begin{equation*}\varepsilon_{R}(1_{\cR})=1_{A},\quad \varepsilon_{R}(s_{R}(\varepsilon_{R}(X))Y)=\varepsilon_{R}(XY)=\varepsilon_{R}(t_{R}(\varepsilon_{R}(X))Y)\end{equation*}
for all $X,Y\in \cR$ and $a\in A$.
\end{itemize}
\end{defi}

\begin{defi}\label{def. full Hopf algebroid1}
    A left $B$-bialgebroid $(\cL, \varepsilon_{L}, \Delta_{L}, s_{L}, t_{L}, m)$ is a full Hopf algebroid, if there is an invertible anti-algebra map $S:\cL\to \cL$, such that
    \begin{itemize}
        \item [(i)] $S\circ t_{L}=s_{L}$,
        \item [(ii)] $\one {(S^{-1}\two{X})} \ot \two {(S^{-1}\two{X})}\one{X}  = S^{-1}(X) \ot 1_\cL$
        \item [(iii)] $\one {S(\one{X})} \two{X} \ot \two{S(\one{X})}  = 1_\cL \ot S(X)$.
    \end{itemize}
\end{defi}

There is another equivalent definition of full Hopf algebroids with bijective antipodes given by \cite{BS}

\begin{defi}\label{def. full Hopf algebroid2}
    A full Hopf algebroid consists of a left bialgebroid $(\cL, \varepsilon_{L}, \Delta_{L}, s_{L}, t_{L}, m)$ over $B$ and a right bialgebroid $(\cR, \varepsilon_{R}, \Delta_{R}, s_{R}, t_{R}, m)$ over $A$, such that 
    \begin{itemize}
        \item [(i)] $A$ is isomorphic to $B^{op}$,  $\cL$ and $\cR$ have the same underlying algebra structure (denoted by $H$);
        \item [(ii)] $s_{L}(B)=t_{R}(A)$ and $t_{L}(B)=s_{R}(A)$ as subrings of $H$;
        \item [(iii)] $(\id\di\Delta_{R})\circ\Delta_{L}=(\Delta_{L}\blacklozenge_{A}\id)\circ\Delta_{R}$ and $(\id\blacklozenge_{A}\Delta_{L})\circ\Delta_{R}=(\Delta_{R}\di\id)\circ\Delta_{L}$;
        \item [(iv)] $\cL$ is a left Hopf and anti-left Hopf algebroid. 
    \end{itemize}
\end{defi}
\begin{rem}
Given a Hopf algebroid  $(\cL, \varepsilon_{L}, \Delta_{L}, s_{L}, t_{L}, m)$ with $S$ as in Definition \ref{def. full Hopf algebroid1}, we can construct a right bialgebroid $(\cR, \varepsilon_{R}, \Delta_{R}, s_{R}, t_{R}, m)$ over $A$ with $A:=B^{op}$ and $\cR:=\cL$ as algebra; the source and target maps are given by

\begin{align}\label{equ. left source target maps to right source and target maps}
  s_{R}(a):=t_{L}(a), \quad t_{R}(a):=S^{-1}\circ t_{L}(a);  
\end{align}
the coproduct is given by
\begin{align}\label{equ. left coproduct to right coproduct}
    \Delta_{R}(X):=S(S^{-1}(X)\t)\ot S(S^{-1}(X)\o)=S^{-1}(S(X)\t)\ot S^{-1}(S(X)\o);
\end{align}
the counit is given by
\begin{align}\label{equ. left counit to right counit}
\varepsilon_{R}:=\varepsilon_{L}\circ S.
\end{align}
We can also see $\cL$ is a left Hopf algebroid and an anti-left Hopf algebroid, with 
\begin{align}\label{equ. Galois map of full Hopf}   X_{+}\ot_{\BB}X_{-}=X\teins{}\ot_{\BB}S(X\tzwei{}),\quad X_{[+]}\ot_{B}X_{[-]}=X\tzwei{}\ot_{B}S^{-1}(X\teins{}).
\end{align}
We also have
\begin{align}\label{equ. anti-colinear}
    S^{\pm}(X)\teins{}\ot S^{\pm}(X)\tzwei{}&=S^{\pm}(X\t)\ot S^{\pm}(X\o),\\
    S^{\pm}(X)\o\ot S^{\pm}(X)\t&=S^{\pm}(X\tzwei{})\ot S^{\pm}(X\teins{}).
\end{align}
\end{rem}
 We will always use $(\cL, \cR, S)$ to denote the full Hopf algebroid with the left $B$-bialgebroid $\cL$ and right $A$-bialgebroid $\cR$ with $H$ being the underlying algebra of $\cL$ and $\cR$ as in Definition \ref{def. full Hopf algebroid2}, equipped with an invertible antipode $S$ in the sense of \Cref{def. full Hopf algebroid1}. Moreover, in order to simplify the relation in $(\cL, \cR, S)$, with no loss of generality, we will always assume the base algebra of $\cL$ is the opposite base algebra of $\cR$, i.e. $A=B^{op}$, with their structure satisfies (\ref{equ. left source target maps to right source and target maps}) and (\ref{equ. left coproduct to right coproduct}).

 By \cite{BB}, a comodule of $(\CL,\CR, S)$ is defined to be a comodule of $\CL$. Here we give the definition of cleft extensions of full Hopf algebroids with bijective antipode:
 \begin{defi}\cite{BB}\label{def. cleft extension by BB}
     Let $(\CL,\CR, S)$ be a full Hopf algebroid over $B$ with bijective antipode. A cleft extension is a right $\CL$-comodule algebra $P$ such that 
      \begin{itemize}
        \item [(i)] $B\subseteq P^{\co\CL}$ as subalgebra,
 \item[(ii)]  There is a right $\cL$-colinear map $j:\CL\to P$  such that $j(s_{L}(b)\,X\,s_{R}(\ol a))=bj(X)\ol a$ and a map  $j^{c}:\CL\to P$ such that $j^{c}(t_{L}(b)\,X\,t_{R}(\ol a))=\ol a \,j^{c}(X)b$. Moreover, 
 \[j(X\teins{})j^{c}(X\tzwei{})=\varepsilon_{L}(X)\,\qquad j^{c}(X\o)j(X\t)=\overline{\varepsilon_{R}(X)}, \]
 for any $a,b\in B$ and $X\in \CL$.
 \end{itemize}
 \end{defi}

\begin{thm}
    If $(\CL,\CR, S)$ is a Hopf algebroid  with bijective antipode, then \Cref{def. cleft extension by BB} and \Cref{Anti-cleft extensions} are equivalent.
\end{thm}
\begin{proof}
 With the help of antipode, every comodule is  reversible.   Assume we have $(j,j^c)$ as in  \Cref{def. cleft extension by BB} and we view $\cL$ as a Hopf algebroid by (\ref{equ. Galois map of full Hopf}), we define $\trho:=j$ and $\rho:=j^{c}\circ S^{-1}$. By definition $\trho$ is right $\cL$-colinear and satisfies $\trho(s_{L}(b)X)=b\trho(X)$. We can also see $\rho(Xs_{L}(b))=j^{c}\circ S^{-1}(Xs_{L}(b))=j^{c}(t_{L}(b)S^{-1}(X))=j^{c}(S^{-1}(X))b=\rho(X)b$. Now, by \Cref{equ. Galois map of full Hopf} we can check 
     \begin{align*}
         \trho(X_{+})\rho(X_{-})=& \trho(X\teins{})\rho(S(X\tzwei{}))
         =j(X\teins{})j^{c}(X\tzwei{})=\varepsilon_{L}(X);
     \end{align*}
     \begin{align*}
         \rho(X_{[+]})\trho(X_{[-]})=&\rho(X\tzwei{})\trho(S^{-1}(X\teins{}))=j^{c}( S^{-1}(X\tzwei{}))j(S^{-1}(X\teins{}))\\
         =&j^{c}( S^{-1}(X)\o)j(S^{-1}(X)\t)=\overline{\varepsilon_{R}(S^{-1}(X))}=\overline{\varepsilon_{L}(X)},
     \end{align*}
     where the 3rd step uses (\ref{equ. anti-colinear}) and the 5th step uses (\ref{equ. left counit to right counit}). 

     Conversely, assume $(\rho,\trho)$ is the cleaving map in \Cref{Anti-cleft extensions}. Define $j:=\trho$ and $j^{c}:=\rho\circ S$. As $t_{L}=s_{R}$, we have $j(s_{L}(b)\,X\,s_{R}(\ol a))=\trho(s_{L}(b)\,X\,t_{L}(\ol a))=bj(X)\ol a$. Also, $j^{c}(t_{L}(b)\,X\,t_{R}(\ol a))=\rho(S(t_{L}(b)\,X\,t_{R}(\ol a)))=\rho(t_{L}(\ol a)S(X)s_{L}(b))=\ol a \,j^{c}(X)b$. Moreover, as above $j(X\teins{})j^{c}(X\tzwei{})=\trho(X_{+})\rho(X_{-})=\varepsilon_{L}(X)$ Also,
     \begin{align*}
         j^{c}(X\o)j(X\t)=&\rho(S(X\o))\trho(X\t)=\rho(S(X)\tzwei{})\trho(S^{-1}(S(X)\teins{}))\\
         =&\rho(S(X)_{[+]})\trho(S(X)_{[-]})=\overline{\varepsilon_{L}(S(X))}=\overline{\varepsilon_{R}(X)}.
     \end{align*}
 \end{proof}

\end{document}